\documentclass[12pt]{article}
\input{style}
\usepackage[color = orange!10]{todonotes}
\usepackage{enumitem, varwidth} 

\begin{document}

\title{Backward martingale transport maps and equilibrium with
  insider}

\author{Dmitry Kramkov\footnote{Carnegie Mellon University, Department
    of Mathematical Sciences, 5000 Forbes Avenue, Pittsburgh, PA,
    15213-3890, USA,
    \href{mailto:kramkov@cmu.edu}{kramkov@cmu.edu}. The author also
    has a research position at the University of Oxford.}  and Mihai
  S\^{i}rbu\footnote{The University of Texas at Austin, Department of
    Mathematics, 2515 Speedway Stop C1200, Austin, Texas 78712,
    \href{mailto:sirbu@math.utexas.edu}{sirbu@math.utexas.edu}. The
    research of this author was supported in part by the National
    Science Foundation under Grant DMS 1908903.}  }

\date{\today}

\maketitle

\begin{abstract}
  We consider an optimal transport problem with backward martingale
  constraint. The objective function is given by the scalar product of
  a pseudo-Euclidean space $S$.  We show that the supremums over maps
  and plans coincide, provided that the law $\nu$ of the input random
  variable $Y$ is atomless.  An optimal map $X$ exists if $\nu$ does
  not charge any $c-c$ surface (the graph of a difference of convex
  functions) with strictly positive normal vectors in the sense of the
  $S$-space.  The optimal map $X$ is unique if $\nu$ does not charge
  $c-c$ surfaces with nonnegative normal vectors in the $S$-space. As
  an application, we derive sharp conditions for the existence and
  uniqueness of equilibrium in a multi-asset version of the model with
  insider from \citet{RochVila:94}. In the linear-Gaussian case, we
  characterize Kyle's lambda, the sensitivity of price to trading
  volume, as the unique positive solution of a non-symmetric algebraic
  Riccati equation.
\end{abstract}

% \listoftodos

\begin{description}
\item[Keywords:] martingale optimal transport, pseudo-Euclidean space,
  Kyle's lambda, maximal monotone multifunction, Riccati equation.
\item[AMS Subject Classification (2010):] 60G42, 91B24, 91B52.
  % \item[JEL Classification:]
\end{description}

\section{Introduction}
\label{sec:introduction}

Let $S$ be a symmetric invertible $d\times d$ matrix with
$m\in\braces{0,1,\dots d}$ positive eigenvalues. We interpret the
bilinear form
\begin{displaymath}
  S(x,y)\set \ip{x}{Sy}=\sum_{i,j=1}^d x^i S_{ij}y^j , \quad x,y\in \real{d},
\end{displaymath}
as the scalar product of a pseudo-Euclidean space, called the
$S$-space.

Let $Y$ be a $d$-dimensional random variable with finite second
moments defined on a probability space
$(\Omega, \mathcal{F}, \mathbb{P})$.  We consider the Monge-type
optimal transport problem:
\begin{displaymath}
  \text{maximize~} \frac12 \EP{S(X,Y)} \text{~over~}
  X\in\mathcal{X}(Y), 
\end{displaymath} 
where $\mathcal{X}(Y)$ is the family of backward martingale
\emph{maps}:
\begin{displaymath}
  \mathcal{X}(Y) \set \descr{X}{X \text{~is 
      $Y$-measurable and~} \cEP{Y}{X} = X}.  
\end{displaymath}

The Kantorovich-type relaxation of this problem is to maximize the
same objective over backward martingale \emph{plans}. By possibly
enlarging the probability space, we can represent such plans as the
joint laws of $(X,Y)$ for random variables $X$ satisfying the
martingale constraint: $\cEP{Y}{X} = X$.  The maps have the additional
$Y$-measurability property: $X=f(Y)$ for some Borel function $f$.  The
plan problem is easier to study due to the convexity of the
optimization set.

The relations between the map and plan problems and their properties
depend on the regularity of the law $\nu$ of the input random variable
$Y$.
\begin{enumerate}[label = {\rm (\alph{*})}, ref={\rm (\alph{*})}]
\item \label{item:1} Theorem~\ref{th:1} shows that the map and plan
  problems have the same values if the law of $\nu$ is
  \emph{atomless}.
\item \label{item:2} Theorem~\ref{th:4} shows that an optimal map
  exists if $\nu$ does not charge any $c-c$ surface (the graph of a
  difference of convex functions) with \emph{strictly positive} normal
  vectors in the $S$-space.
\item \label{item:3} Theorem~\ref{th:6} shows that the optimal map is
  unique if, in addition to the assumption of~\ref{item:2}, $\nu$ does
  not charge $c-c$ surfaces whose normal vectors in the $S$-space are
  \emph{isotropic} for $m=1$ and \emph{nonnegative and almost
    isotropic} for $m>1$.
\end{enumerate}
We point out that the assumptions of \ref{item:1}, \ref{item:2},
and~\ref{item:3} hold if $\nu$ has a density with respect to Lebesgue
measure.  For $d=2$ and the \emph{standard} matrix $S =
\begin{pmatrix}
  0 & 1 \\
  1 & 0
\end{pmatrix}$, Theorems~\ref{th:4} and~\ref{th:6} improve the
existence and uniqueness criteria from~\citet{KramXu:22}, where the
covering in items~\ref{item:2} and~\ref{item:3} has been accomplished
with Lipschitz (not $c-c$) surfaces.

The proof of Theorem~\ref{th:1} from item~\ref{item:1} is based on a
result of independent interest, the \emph{pointwise uniform}
approximation of plans by maps.  Let $X$ and $Y$ be random variables
and assume that the law of $Y$ is atomless. For every $\epsilon>0$,
Theorem~\ref{th:12} constructs a random variable $Z$ having the same
law as $Y$ and such that $X$ is $Z$-measurable and
$\abs{Z-Y}<\epsilon$. The novelty of this construction is that we fix
the ``target'' $X$ and modify $Y$. In a more traditional approach, as
in~\citet{Prat:07} and~\citet{BeigLack:18}, it is the ``origin'' $Y$
that remains unchanged. As a consequence, only \emph{in law}
approximation of plans by maps is possible.

The original motivation for the backward martingale transport comes
from Kyle's equilibrium for insider trading introduced
in~\citet{Kyle:85}. The paper~\cite{KramXu:22} studies a version of
such equilibrium from~\citet{RochVila:94} and shows its connection to
the map problem for $d=2$ and the standard matrix $S$.  In
Section~\ref{sec:equil-with-insid}, we investigate the multi-asset
version of Rochet and Vila's equilibrium, where $d=2m > 2$.
Theorem~\ref{th:8} shows that an equilibrium with a monotone pricing
function exists if and only if one can find an optimal map $X$ and a
dual optimizer $G$ such that the law of $(X,Y)$ is an optimal plan and
the projection of $G$ on the first $m$ coordinates is the whole space
$\real{m}$. If $Y$ is a Gaussian random variable, then the equilibrium
and map problems have explicit linear solutions described in
Theorem~\ref{th:11}. In particular, Theorem~\ref{th:11} characterizes
the multi-dimensional analogue of Kyle's lambda from~\cite{Kyle:85},
the sensitivity of price to trading volume, as the unique positive
matrix solving a non-symmetric algebraic Riccati equation.

\subsection*{Notations}
\label{sec:notations}

The scalar product and the norm in the Euclidean space $\real{d}$ are
written as
\begin{displaymath}
  \ip{x}{y} \set \sum_{i=1}^d x_i y_i, \quad \abs{x} \set
  \sqrt{\ip{x}{x}}, \quad x,y\in \real{d}.
\end{displaymath}

A multifunction $\mmap{T}{\real{m}}{\real{n}}$ is a mapping from
$\real{m}$ into subsets of $\real{n}$. The domain of $T$ is the set of
those $x\in \real{m}$ where $T(x)$ is not empty:
\begin{displaymath}
  \dom{T} \set \descr{x\in \real{m}}{T(x) \not=\emptyset}. 
\end{displaymath}

For a Borel probability measure $\mu$ on $\real{d}$, a
$\mu$-integrable $m$-dimensional Borel function $f=(f_1,\dots,f_m)$,
and an $n$-dimensional Borel function $g = (g_1,\dots,g_n)$, the
notation $\mu(f|g)$ stands for the $m$-dimensional vector of
conditional expectations of $f_i$ given $g$ under $\mu$:
\begin{displaymath}
  \mu(f|g) = (\mu(f_1|g_1,\dots,g_n),\dots,\mu(f_m|g_1,\dots,g_n)).
\end{displaymath}
In particular, we write $\mu(f)$ for the vector of expected values:
\begin{displaymath}
  \mu(f) = \int f d\mu = (\int f_1 d\mu, \dots, \int f_m d\mu) =
  (\mu(f_1),\dots,\mu(f_m)). 
\end{displaymath}
We write $\supp{\mu}$ for the support of $\mu$, the smallest closed
set of full measure.

Similarly, if $(\Omega, \mathcal{F}, \mathbb{P})$ is a probability
space, $X$ and $Y$ are respectively, $m$- and $n$-dimensional random
variables, and $Y$ is integrable, then
\begin{displaymath}
  \cEP{Y}{X} = \cbraces{\cEP{Y_1}{X_1,\dots,X_m}, \dots,
    \cEP{Y_n}{X_1,\dots,X_m}}, 
\end{displaymath}
denotes the $n$-dimensional vector of conditional expectations of
$Y_i$ with respect to $X$. All relations between random variables are
understood in the $\as{\mathbb{P}}$ sense. In particular, $X$ is
$Y$-measurable if and only if $X=f(Y)$ ($\as{\mathbb{P}}$) for a Borel
function $\map{f}{\real{m}}{\real{d}}$.

\section{Equality of values of plan and map problems}
\label{sec:equality-values-plan}

We denote by $\msym{d}{m}$ the family of symmetric $d\times d$
matrices of full rank with $m \in \braces{0,1,\dots,d}$ positive
eigenvalues. For $S \in \msym{d}{m}$, the bilinear form
\begin{displaymath}
  S(x,y) \set \ip{x}{Sy}  = \sum_{i,j=1}^d x^i S_{ij}y^j, \quad
  x,y\in \real{d},
\end{displaymath}
defines the scalar product on a pseudo-Euclidean space $\real{d}_m$
with dimension $d$ and index $m$, which we call the $S$-space. The
quadratic form $S(x,x)$ is called the \emph{scalar square}; its value
may be negative.

Let $(\Omega, \mathcal{F}, \mathbb{P})$ be a probability space and $Y$
be a $d$-dimensional random variable with finite second moment:
$Y\in \lsp{2}{d}$.  Our goal is to
\begin{equation}
  \label{eq:1}
  \text{maximize} \quad \frac12 \EP{S(X,Y)}
  \quad \text{over}\quad X \in \mathcal{X}(Y), 
\end{equation} 
where
\begin{displaymath}
  \mathcal{X}(Y) \set \descr{X\in
    \lsp{2}{d}}{X \text{~is 
      $Y$-measurable and~} \cEP{Y}{X} = X}.  
\end{displaymath}

The Kantorovich-type relaxation of the optimal \emph{map}
problem~\eqref{eq:1} is the optimal \emph{plan} problem:
\begin{equation}
  \label{eq:2}
  \text{maximize} \quad \frac12 \int S(x, y)
  d\gamma\quad\text{over}\quad 
  \gamma\in \Gamma(\nu), 
\end{equation}
where $\nu \set \law{Y}$ belongs to $\ps{2}{\real{d}}$, the family of
Borel probability measures on $\real{d}$ with finite second moments,
and
\begin{displaymath}
  \Gamma(\nu) \set \descr{\gamma \in
    \ps{2}{\real{2d}}}{\gamma(\real{d},dy) =  
    \nu(dy) \text{~and~}
    \gamma(y|x)=x}.  
\end{displaymath}
The plan problem is easier to study, because the optimization set
$\Gamma(\nu)$ is a convex compact set in the Wasserstein $2$-space and
thus, an optimal plan always exists. We refer to Lemma~2.8 and
Theorem~2.5 in~\citet{KramSirb:24} for the details. Clearly,
$\law{X,Y} \in \Gamma(\nu)$ for every $X\in
\mathcal{X}(Y)$. Therefore,
\begin{displaymath}
  \sup_{X \in \mathcal{X}(Y)} \EP{S(X,Y)} \leq
  \max_{\gamma\in \Gamma(\nu)} \int S(x,y) d\gamma.
\end{displaymath}
Notice that the inequality may be strict and an optimal map may not
exist, as Examples~5.2 and~5.3 in~\cite{KramXu:22} show.

The following theorem is similar to that of~\cite{Prat:07} obtained
for the classical unconstrained optimal transport problem. For $d=2$
and the \emph{standard} $S =
\begin{pmatrix}
  0 & 1 \\
  1 & 0
\end{pmatrix} \in \msym{2}{1}$, it has been proved in~\cite{KramXu:22}.

\begin{Theorem}
  \label{th:1}
  Let $S\in \msym{d}m$, $Y\in \lsp{2}{d}$, and suppose that
  $\nu \set \law{Y}$ is atomless. Then the problems~\eqref{eq:1}
  and~\eqref{eq:2} have the same values:
  \begin{displaymath}
    \sup_{X \in \mathcal{X}(Y)} \EP{S(X,Y)} =
    \max_{\gamma\in \Gamma(\nu)} \int S(x,y) d\gamma.
  \end{displaymath}
\end{Theorem}

\begin{proof}
  Let $\gamma$ be an optimal plan for~\eqref{eq:2}. By extending, if
  necessary, the underlying probability space we can assume that
  $\gamma = \law{X,Y}$ for some random variable $X$. As
  $\gamma(y|x) = x$, we have that $X = \cEP{Y}{X}$.

  Let $\epsilon>0$. Theorem~\ref{th:12} yields a $d$-dimensional
  random variable $Z=Z(\epsilon)$ such that
  \begin{displaymath}
    \law{Z} = \law{Y}, \quad \abs{Z-Y} \leq \epsilon, \quad  X
    \text{~is $Z$-measurable}.  
  \end{displaymath}
  Since $X$ is $Z$-measurable, the conditional expectation
  $V \set \cEP{Z}{X}$ is $Z$-measurable as well. Thus, there is a
  Borel function $\map{f}{\real{d}}{\real{d}}$ such that $V = f(Z)$.
  We clearly have $V=\cEP{Z}{V}$.  Since $Y$ and $Z$ have identical
  laws,
  \begin{displaymath}
    U \set f(Y) = \cEP{Y}{U}. 
  \end{displaymath}
  For all $x,y,v,z$ in $\real{d}$, we have that
  \begin{align*}
    \abs{S(x,y) - S(v,z)} & \leq \abs{S(x-v,y)} + \abs{S(v,z-y)} \\
                          & \leq \norm{S}(\abs{y}\abs{x-v} +
                            \abs{v}\abs{z-y}),  
  \end{align*}
  where $\norm{S} \set \max_{\abs{x}=1}\abs{Sx}$, the norm of $S$.
  Since $\abs{Z-Y} \leq \epsilon$ and
  \begin{align*}
    \abs{V - X} & = \abs{\cEP{Z - Y}{X}} \leq 
                  \cEP{\abs{Z-Y}}{X} \leq 
                  \epsilon, 
  \end{align*}
  we have that
  \begin{align*}
    \abs{S(X, Y) - S(V,Z)}
    \leq \epsilon \norm{S} (\abs{Y} + \abs{V}).                 
  \end{align*}
  As
  $\EP{\abs{V}} = \EP{\abs{\EP{Z|X}}} \leq \EP{\abs{Z}} \leq
  \EP{\abs{Y}} + \epsilon$ and $\law{U,Y} = \law{V,Z}$, we obtain that
  \begin{displaymath}
    \int S(x,y) d\gamma  = \EP{S(X,Y)} \leq
    \EP{S(U, Y)} + \epsilon \norm{S} (2\EP{\abs{Y}} +
    \epsilon). 
  \end{displaymath}
  The result follows, because $U\in \mathcal{X}(Y)$ and $\epsilon$ is
  any positive number.
\end{proof}

\section{Optimal plans and dual problem}
\label{sec:dual-problem}

Let $S\in \msym{d}m$, $\nu \in \ps{2}{\real{d}}$, and
$\gamma \in \Gamma(\nu)$.  Theorem~\ref{th:2}\ref{item:6} contains a
new necessary and sufficient condition for $\gamma$ to be an optimal
plan for~\eqref{eq:2}. The result complements Theorem~2.5
in~\cite{KramSirb:24} and is the starting point of the present work.

We begin by introducing some concepts and notations.  A set
$G\subset \real{d}$ is called $S$-\emph{monotone} or
$S$-\emph{positive} if
\begin{displaymath}
  S(x-y,x-y)  \geq 0, \quad x, y
  \in G.   
\end{displaymath}
An $S$-monotone set $G$ is called \emph{maximal} if it is not a strict
subset of an $S$-monotone set.

\begin{Example}[Standard form]
  \label{ex:1}
  If $d=2m$ and
  \begin{displaymath}
    S(x,y) = \sum_{i=1}^m \cbraces{x^i y^{m+i} + x^{m+i}y^i},  \quad
    x,y\in \real{2m}, 
  \end{displaymath}
  then $S \in \msym{2m}m$ and the $S$-monotonicity means the standard
  monotonicity in $\real{2m}=\real{m}\times \real{m}$.
\end{Example}

It has been shown in~\cite[Theorem~2.5]{KramSirb:24}, that a dual
problem to~\eqref{eq:2} is to
\begin{equation}
  \label{eq:3}
  \text{minimize} \quad \EP{\psi_G(Y)} = \int \psi_G(y) d\nu
  \quad\text{over}\quad G \in \mset{S}, 
\end{equation}
where $\mset{S}$ is the family of all maximal $S$-monotone sets and
\begin{displaymath}
  \psi_G(y) \set  \sup_{x\in G} \cbraces{S(x,y) - \frac12 S(x,x)}, \quad 
  y\in \real{d}, 
\end{displaymath}
is the \emph{Fitzpatrick} function in the $S$-space. We refer the
reader to~\cite[Appendix~A]{KramSirb:24} for the basic facts about the
Fitzpatrick functions in the $S$-space. Theorem~2.5
in~\cite{KramSirb:24} shows that an optimal set for~\eqref{eq:3}
always exists and
\begin{displaymath}
  \max_{\gamma \in \Gamma(\nu)} \frac12 \int S(x,y) d\gamma = \min_{G
    \in \mset{S}} \int \psi_G(y) d\nu. 
\end{displaymath}

Let $G\in \mset{S}$. We denote by $P_G$ the projection on $G$ in the
$S$-space:
\begin{equation}
  \label{eq:4}
  \begin{split}
    P_G(y) & \set \argmin_{x\in G} S(x-y,x-y) \\
           & \; = \argmax_{x\in G}
             \cbraces{S(x,y) - \frac12 S(x,x)}, \quad  y\in \real{d}.  
  \end{split}
\end{equation}
Geometrically, $x\in P_G(y)$ if and only if the hyperboloid
\begin{displaymath}
  H_G(y) \set \descr{z \in \real{d}}{S(z,y) - \frac12 S(z,z) =
    \psi_G(y)} 
\end{displaymath}
is \emph{tangent} to $G$ at $x$. If $x\in P_G(y)$, then the vector
$y-x$ is $S$-\emph{regular normal} to $G$ at $x$ in the sense that
\begin{displaymath}
  \limsup_{z\to x, z\in G} \frac{S(y-x,z-x)}{\abs{z-x}} \leq 0. 
\end{displaymath}

\begin{Theorem}
  \label{th:2}
  Let $S \in \msym{d}{m}$ and $\nu \in \ps{2}{\real{d}}$.  For any
  $\gamma \in \Gamma(\nu)$ and $G\in \mset{S}$, the following
  conditions are equivalent:
  \begin{enumerate}[label = {\rm (\alph{*})}, ref={\rm (\alph{*})}]
  \item \label{item:4} $\gamma$ is an optimal plan for~\eqref{eq:2}
    and $G$ is an optimal set for~\eqref{eq:3}.
  \item \label{item:5} $x \in P_G(y)$, $(x,y)\in \supp{\gamma}$.
  \item \label{item:6} $x \in P_G(x) \subset P_G(y)$, $\as{\gamma}$.
  \end{enumerate}
\end{Theorem}

The equivalence of items~\ref{item:4} and~\ref{item:5} has been
already established in Theorem~2.5 in
\cite{KramSirb:24}. Item~\ref{item:6} is new.

The proof of Theorem~\ref{th:2} relies on some lemmas. We first verify
the measurability condition used implicitly in item~\ref{item:6}. We
recall that an $F_\sigma$-set is a countable union of closed sets.

\begin{Lemma}
  \label{lem:1}
  Let $S \in \msym{d}{m}$ and $G\in \mset{S}$. Then
  \begin{gather*}
    \graph{P_G} \set
    \descr{(x,y)}{y\in P_G(x)} \text{~is closed}, \\
    \graph{P_G^{-1}} \set
    \descr{(x,y)}{x\in P_G(y)} \text{~is closed}, \\
    B  \set \descr{(x,y)}{x\in P_G(y), \; P_G(x) \not \subset
      P_G(y)} \text{~is an $F_\sigma$-set}. 
  \end{gather*}
  In particular,
  \begin{displaymath}
    U  \set \descr{(x,y)}{x\in P_G(x) \subset P_G(y)} =
    \graph{P_G^{-1}}\setminus B
  \end{displaymath}
  is a Borel set in $\real{2d}$.
\end{Lemma}

\begin{proof}
  Direct arguments show that $P_G$ has a closed graph. Then,
  trivially, the graph of the inverse multifunction $P_G^{-1}$ is also
  closed.

  We can write $B = \cup_n B_n$, where $B_n$ consists of those
  $(x,y) \in \graph{P_G^{-1}}$ for which there exists $z \in P_G(x)$
  such that
  \begin{displaymath}
    \frac1n \leq \abs{x-z} \leq n, \quad
    S(x,y)-\frac12S(x,x)
    =\psi_G(y)\geq S(z,y)-\frac12S(z,z) +\frac1n.  
  \end{displaymath}
  Elementary arguments show that $B_n$ is a closed set. Hence, $B$ is
  an $F_\sigma$-set.
\end{proof}

For $x\in G$, we denote by $Q_G(x)$ the largest closed convex subset
of $P^{-1}_G(x) \set \descr{y}{x \in P_G(y)}$ whose relative interior
contains $x$. By Lemma~2.12 in \cite{KramSirb:24}, $y \in Q_G(x)$ if
and only if there exist $z \in P^{-1}_G(x)$ and $t\in (0,1)$ such that
$x = ty + (1-t) z$.  If $y\in Q_G(x)$, then the vector $y-x$ is
$S$-\emph{orthogonal} to $G$ at $x$ in the sense that
\begin{displaymath}
  \lim_{u\to x, u\in G} \frac{S(y-x,u-x)}{\abs{u-x}} = 0. 
\end{displaymath}
Theorem~2.5 in \cite{KramSirb:24} shows that $\gamma \in \Gamma(\nu)$
is an optimal plan for~\eqref{eq:2} and $G \in \mset{S}$ is an optimal
set for~\eqref{eq:3} if and only if
\begin{equation}
  \label{eq:5}
  x \in G \text{~and~} y \in Q_G(x), \; \as{\gamma}.
\end{equation}

\begin{Lemma}
  \label{lem:2}
  Let $S\in \msym{d}m$ and $G\in \mset{S}$. If $x \in G$ and
  $y \in Q_G(x)$, then $P_G(x) \subset P_G(y)$.
\end{Lemma}

\begin{proof}
  Let $v \in Q_G(x)$ and $z \in P_G(x)$.  Then $x \in P_G(v)$ and
  \begin{align*}
    0 & \leq \psi_G(v) - \cbraces{S(z,v) - \frac12 S(z,z)} \\
      & = S(x,v) - \frac12 S(x,x) - S(z,v) + \frac12 S(z,z)
    \\
      & = \frac12 S(x-z,x-z) + S(x-z,v-x) = S(x-z,v-x).  
  \end{align*}
  Choosing $v = x \pm \epsilon (y-x)$ for some $\epsilon \in (0,1)$,
  which is possible by the construction of $Q_G(x)$, we obtain
  \begin{displaymath}
    S(x-z,y-x) = 0. 
  \end{displaymath}
  Taking $v=y$, we deduce
  \begin{displaymath}
    \psi_G(y) =  S(z,y) - \frac12 S(z,z). 
  \end{displaymath}
  Hence, $z\in P_G(y)$, as required.
\end{proof}

\begin{proof}[Proof of Theorem~\ref{th:2}.]
  Theorem~2.5 in \cite{KramSirb:24} shows the equivalence
  of~\ref{item:4}, \ref{item:5}, and~\eqref{eq:5}. The result now
  follows from the implications:
  \begin{displaymath}
    x\in G, y \in Q_G(x) \implies x \in P_G(x) \subset P_G(y) \implies
    x \in P_G(y), 
  \end{displaymath}
  where the first assertion has been proved in Lemma~\ref{lem:2} and
  the second one is trivial.
\end{proof}

\section{Existence of optimal maps}
\label{sec:exist-solut-map}

Let $S\in \msym{d}m$ and $Y\in \lsp{2}{d}$. By Theorem~\ref{th:1}, the
map and plan problems \eqref{eq:1} and \eqref{eq:2} have identical
values provided that $\nu = \law{Y}$ is atomless. The main results of
this section, Theorems~\ref{th:3} and~\ref{th:4}, show that an optimal
map exists if $\nu(F) = 0$ for every $c-c$ surface (the graph of a
difference of two convex functions) $F$ having strictly positive
normals in the $S$-space.

Let $G\in \mset{S}$ and $P_G$ be its projection multifunction defined
in~\eqref{eq:4}.  A point
$y\in \dom{P_G} \set \descr{y}{P_G(y) \not = \emptyset}$ is called
\emph{singular}, if $P_G(y)$ is not a singleton. We decompose the set
of singular points of $P_G$ as
\begin{align*}
  \Sigma(P_G) &\set \descr{x\in \dom{P_G}}{P_G(x) \text{~is not a
                point}} = \Sigma_0(P_G) \cup
                \Sigma_1(P_G), \\
  \Sigma_0(P_G) & \set \descr{x\in \Sigma(P_G)}{S(y_1-y_2,y_1-y_2)=0
                  \text{~for \emph{all}~} y_1,y_2 \in P_G(x)}, \\
  \Sigma_1(P_G) & \set \descr{x\in \Sigma(P_G)}{S(y_1-y_2,y_1-y_2)>0
                  \text{~for \emph{some}~} y_1,y_2 \in P_G(x)}. 
\end{align*}
Lemma~\ref{lem:3} contains an equivalent description of
$\Sigma_0(P_G)$.  Figure~\ref{fig:1} provides an illustration for the
standard $S \in \msym{2}{1}$.

\begin{figure}
  \centering
  \begin{tikzpicture}[scale = 0.5]
    % \draw[help lines] (-10,-6) grid (8,6);
   
    \draw[dashed, thin] (-3.6,-0.728) node [right] {$x_1$} --(-1,
    -4.472) node [right] {$y$} -- (-7,-2.8) node [above left] {$x_2$};

    \filldraw [black] (-3.6,-0.728) circle [radius=2pt] (-1,-4.472)
    circle [radius=2pt] (-7,-2.8) circle [radius=2pt];

    \draw[domain=-10:-2.12,smooth,variable=\x,black] plot
    ({\x},{-9.9344/(\x+1) - 4.472}); \draw (-2.3,3.016) node [left]
    {$H_G(y)$};

    \draw[thick,domain=-4.15:4.5,smooth,variable=\x,blue] plot
    ({\x},{(\x/3)^3 + 1}); \draw (4, 3.37) node [above left] {$G$};

    \draw[thick,domain=-10:-6,smooth,variable=\x,blue] plot
    ({\x},{-2.8 +0.276*(\x+7) - 0.12 * (\x+7)^2});

    \draw[thick,domain=-6:-4.15,smooth,variable=\x,blue] plot
    ({\x},{-2.64 +0.23*(\x+6)^(2.4)});
    
    \draw[ultra thick] (-0.75,1) node [above] {$z_1$} -- (0.75,1) node
    [above] {$z_2$};
    
    \filldraw [black] (-0.75,1) circle [radius=2pt] (0.75,1) circle
    [radius=2pt];

  \end{tikzpicture}
  \caption{The figure corresponds to the standard $S\in
    \msym{2}{1}$. The hyperbola $H_G(y)$ with focus at $y$ is tangent
    to $G$ at $x_1$ and $x_2$.  The point $y$ is singular and belongs
    to $\Sigma_1(P_G)$. The horizontal segment $[z_1,z_2] \subset G$
    is an $S$-isotropic set contained in $\Sigma_0(P_G)$. The vector
    $y-x_i$ is $S$-regular normal to $G$ at $x_i$, $i=1,2$.}
  \label{fig:1}
\end{figure}

\begin{Example}
  \label{ex:2}
  Let $d=2$ and $S$ be the standard bilinear form from
  Example~\ref{ex:1}:
  \begin{displaymath}
    S(x,y)=S((x_1,x_2),(y_1,y_2)) = x_1y_2 + x_2y_1.
  \end{displaymath}
  As shown in Example~4.9 of~\citet{KramSirb:22b}, for
  $G\in \mset{S}$, the singular set $\Sigma_0(G)$ is a countable union
  of horizontal and vertical line segments of $G$. The singular set
  $\Sigma_1(G)$ is contained in a countable union of the graphs of
  functions
  \begin{displaymath}
    x_2 = h(x_1) = g_1(x_1) - g_2(x_1), \quad x_1 \in \real{},  
  \end{displaymath}
  where $g_1$ and $g_2$ are convex functions and
  $\epsilon \leq - h'(t) \leq 1/\epsilon$, for a constant
  $\epsilon = \epsilon(h)>0$ and all real $t$ where $h$ is
  differentiable. In particular, $h$ and its inverse $h^{-1}$ are
  strictly decreasing Lipschitz functions.
\end{Example}

The projection multifunction $P_G$ takes values in the closed subsets
of $G$. Following~\citet[p.~59]{CastValad:77}, we then denote by
$\sigma\cbraces{P_G}$ the $\sigma$-algebra generated by the pre-images
$P_G^{-1}(U) \set \descr{y}{P_G(y)\cap U \not=\emptyset}$ of open sets
$U\subset \real{d}$.  If $Y$ is a $d$-dimensional random variable,
then naturally,
\begin{align*}
  \sigma(P_G(Y)) & \set \descr{Y^{-1}(A)}{A \in \sigma(P_G)}\\
                 & \;
                   =\sigma\cbraces{\descr{\omega}{P_G(Y(\omega))\cap
                   U\not=\emptyset}, \,
                   U\text{~is an open set in~}\real{d}}.
\end{align*}
Lemma~\ref{lem:5} shows that every $A\in \sigma(P_G)$ is a Borel set
in $\real{d}$. It follows that $\sigma(P_G(Y))$ is a
sub-$\sigma$-algebra of $\sigma(Y)$ and a conditional expectation with
respect to $\sigma(P_G(Y))$ is a $Y$-measurable random variable.
Lemma~\ref{lem:5} also proves that the singular sets $\Sigma_0(P_G)$
and $\Sigma_1(P_G)$ belong to $\sigma(P_G)$. In particular, they are
Borel sets in $\real{d}$.

We recall that an optimal set $G\in \mset{S}$ for the dual
problem~\eqref{eq:3} always exists.

\begin{Theorem}
  \label{th:3}
  Let $S\in \msym{d}{m}$, $Y \in \mathcal{L}_2(\real{d})$, and denote
  $\nu \set \law{Y}$.  Let $G\in \mset{S}$ be an optimal set
  for~\eqref{eq:3} and assume that $Y \not\in \Sigma_1(P_G)$, that is,
  $\nu \cbraces{\Sigma_1(P_G)}=0$. Then
  \begin{displaymath}
    X \set \cEP{Y}{P_G(Y)} \set \cEP{Y}{\sigma(P_G(Y))}
  \end{displaymath}
  is an optimal map for~\eqref{eq:1} and the law of $(X,Y)$ is an
  optimal plan for~\eqref{eq:2}.

  If $Z \in \mathcal{L}_2(\real{d})$ and the law of $(Z,Y)$ is an
  optimal plan (in particular, if $Z$ is an optimal map), then
  \begin{gather}
    \label{eq:6}
    Z \in P_G(Z) = P_G(Y), \\
    \label{eq:7}
    \text{$X$ is $Z$-measurable, $X = \cEP{Z}{X}$, and $S(Z-X,Z-X) = 0$.}
  \end{gather}
\end{Theorem}

Relations~\eqref{eq:6} and~\eqref{eq:7} show that all optimal maps
take values in $G$, have the same $S$-projection on $G$ as $Y$, and
that $X$ generates the smallest $\sigma$-algebra among them.

The direct use of Theorem~\ref{th:3} requires the knowledge of a dual
minimizer $G$. An obvious sufficient condition is to assume that
$Y\not\in \Sigma_1(P_G)$ for \emph{every} $G\in \mset{\real{d}}$ such
that $\EP{\psi _G(Y)}<\infty$. A stronger, but more explicit
sufficient condition is stated in Theorem~\ref{th:4}.

The proof of Theorem~\ref{th:3} relies on some lemmas.

\begin{Lemma}
  \label{lem:3}
  Let $S\in \msym{d}m$, $G\in \mset{S}$, and $y\in \dom{P_G}$. Then
  $P_G(y)$ is a convex set if and only if
  \begin{displaymath}
    S(u-v,u-v)= 0, \quad u,v \in P_G(y). 
  \end{displaymath}
  In other words, the singular set $\Sigma_0(P_G)$ admits the
  equivalent description:
  \begin{align*}
    \Sigma_0(P_G) = \descr{y\in \Sigma(P_G)}{P_G(y) \text{~is a convex
    set}}. 
  \end{align*}
\end{Lemma}
\begin{proof}
  Let $u,v \in P_G(y)$ and denote $w = \frac12(u+v)$. We have that
  \begin{align*}
    \psi_G(y) & = S(y,u) - \frac12 S(u,u) = S(y,v) - \frac12 S(v,v) \\
              & = S(y,w) - \frac12 S(w,w) - \frac14 \cbraces{S(u,u) +
                S(v,v) - 2 S(w,w)} \\
              & = S(y,w) - \frac12 S(w,w) - \frac18 S(u-v,u-v).
  \end{align*}
  In particular,
  \begin{displaymath}
    \psi_G(y) =  S(y,w) - \frac12 S(w,w) \iff S(u-v,u-v) = 0. 
  \end{displaymath}

  If $P_G(y)$ is convex, then $w\in P_G(y)$, which implies that
  $S(u-v,u-v) = 0$.

  Assume now that $S(u-v,u-v) = 0$. Direct computations show that
  \begin{displaymath}
    S(x-w,x-w)  = \frac12 \cbraces{S(x-u,x-u) + S(x-v,x-v)} \geq 0,
    \quad x\in G.
  \end{displaymath}
  As $G \in \mset{S}$, we obtain that $w\in G$ and then that
  $w\in P_G(y)$.
\end{proof}

We state next a version of the classical result on measurability of
multifunctions and their measurable selections.

\begin{Lemma}
  \label{lem:4}
  Let $\mmap{T}{\real{m}}{\real{n}}$ be a multifunction whose graph
  \begin{displaymath}
    \graph{T} \set \descr{(u,v)}{u\in \real{m}, v\in T(u)}
  \end{displaymath}
  is closed. Then
  \begin{enumerate}[label = {\rm (\alph{*})}, ref={\rm (\alph{*})}]
  \item \label{item:7} The pre-image
    $T^{-1}(B) \set \descr{u\in \real{m}}{T(u) \cap B \not=
      \emptyset}$ of every $F_\sigma$-set $B$ is an $F_\sigma$-set.
  \item \label{item:8} The domain
    $D \set \descr{u\in \real{m}}{T(u)\not = \emptyset}$ of $T$ is an
    $F_\sigma$-set.
  \item \label{item:9} There exists a Borel function
    $\map{f}{D}{\real{n}}$ such that $f(u) \in T(u)$, $u\in D$.
  \end{enumerate}
\end{Lemma}

\begin{proof}
  As $\graph{T}$ is closed, $T^{-1}(C)$ is a closed set for every
  compact $C$. Since every $F_\sigma$-set $B$ is a countable union of
  compacts $(C_n)$ and
  \begin{displaymath}
    T^{-1}(B) = T^{-1}\Bigl(\bigcup_n C_n\Bigr) = \bigcup_n
    T^{-1}\cbraces{C_n},
  \end{displaymath}
  we obtain~\ref{item:7}.  Taking $B=\real{n}$ we prove~\ref{item:8}.
  As every open set $U$ is an $F_\sigma$-set, we deduce that
  $T^{-1}(U)$ is an $F_\sigma$-set and, in particular, a Borel
  set. Having a closed graph, $T$ takes values in the closed subsets
  of (the complete separable metric space) $\real{d}$. The measurable
  selection theorem from \cite[Theorem~III.6, page~65]{CastValad:77}
  yields~\ref{item:9}.
\end{proof}

\begin{Lemma}
  \label{lem:5}
  Let $S\in \msym{d}m$ and $G\in \mset{S}$. Then
  \begin{enumerate}[label = {\rm (\alph{*})}, ref={\rm (\alph{*})}]
  \item \label{item:10} Every $A\in \sigma(P_G)$ is a Borel set in
    $\real{d}$.
  \item \label{item:11} For every $s\in \real{d}$, the function
    \begin{displaymath}
      g_s(y) \set \sup_{x\in P_G(y)} \ip{s}{x} \in \realextplusminus,
      \quad y \in \real{d}, 
    \end{displaymath}
    is $\sigma(P_G)$-measurable, where we used the usual convention
    that $\sup$ over an empty set is $-\infty$.
  \item \label{item:12} The singular sets $\Sigma_0(P_G)$ and
    $\Sigma_1(P_G)$ belong to $\sigma(P_G)$.
  \end{enumerate}
\end{Lemma}

\begin{proof}
  \ref{item:10}: By Lemma~\ref{lem:1}, the graph of $P_G$ is closed.
  Since any open set $U$ is an $F_\sigma$-set, Lemma~\ref{lem:4} shows
  that $P_G^{-1}(U)$ is an $F_\sigma$-set. In particular, it is a
  Borel set. The result now holds by the definition of $\sigma(P_{G})$
  .
  
  \ref{item:11}: Fix $s\in \real{d}$. For every $a\in \real{}$, we
  have
  \begin{displaymath}
    \descr{y \in \real{d}}{g_s(y)>a}=
    P^{-1}_G \cbraces{\descr{x\in \real{d}}{\ip{s}{x}>a}}
    \in \sigma(P_G).
  \end{displaymath}
  The $\sigma(P_G)$-measurability of $g_s$ readily follows.

  \ref{item:12}: Fix a sequence $x_i$, $i=1,2,\dots$ dense in
  $\real{d}$.  Denote by $r_k$, $k=1,2,\dots$ an enumeration of all
  positive rationals and set $\alpha\set (i, j,k,l)$ and
  \begin{displaymath}
    B^{\alpha}_1\set\descr{x\in \real{d}}{|x-x_i|<r_k}, \quad
    B^{\alpha}_2\set\descr{x\in \real{d}}{|x-x_j|<r_l}.
  \end{displaymath}
  Restrict the set of countable indexes $\alpha= (i, j,k,l)$ to those
  for which
  \begin{displaymath}
    B^{\alpha}_1\cap B^{\alpha}_2=\emptyset.
  \end{displaymath}
  We further denote by $\beta $ the indexes $\alpha$ with the
  additional property:
  \begin{displaymath}
    \inf _{u\in B^{\beta}_1, v\in B^{\beta}_2}S(u-v,u-v)>0.
  \end{displaymath}
  The conclusion follows from the definition of $\sigma (P_G)$, once
  we observe that
  \begin{align*}
    \Sigma(P_G) & =\bigcup _{\alpha} \cbraces{P^{-1}_G(B^{\alpha}_1)\cap 
                  P^{-1}_G(B^{\alpha}_2)},\\
    \Sigma_1(P_G) &=\bigcup _{\beta} \cbraces{P^{-1}_G(B^{\beta}_1)\cap 
                    P^{-1}_G(B^{\beta}_2)}, 
  \end{align*}
  and $\Sigma_0(P_G) = \Sigma(P_G) \setminus \Sigma_1(P_G)$.
\end{proof}

\begin{proof}[Proof of Theorem~\ref{th:3}]
  Let $\gamma$ be an optimal plan for~\eqref{eq:2}. By extending, if
  necessary, the probability space, we can find $Z\in \lsp{2}{d}$ such
  that $\gamma = \law{Z, Y}$. As $x = \gamma(y|x)$, we have
  $Z = \cEP{Y}{Z}$. By Theorem~\ref{th:2}\ref{item:6},
  \begin{align*}
    Z \in P_G(Z) \subset P_G(Y). 
  \end{align*}
  It goes without saying that all pointwise relations are understood
  in the almost sure sense.  Since $Y\not\in \Sigma_1(P_G)$, we have
  \begin{displaymath}
    S(x-y,x-y) = 0, \quad x,y\in P_G(Y), 
  \end{displaymath}
  which readily implies that $P_G(Y) \subset P_G(Z)$. We have
  proved~\eqref{eq:6}.
  
  As $P_G(Z) = P_G(Y)$, the $\sigma$-algebras generated by $P_G(Y)$
  and $P_G(Z)$ differ only by $\mathbb{P}$-null sets.  Therefore, the
  conditioning on either of them yields the same result. In addition,
  $\sigma(P_G(Z))\subset \sigma(Z)$.  By the tower property of
  conditional expectations,
  \begin{align*}
    X  & \set   \cEP{Y}{P_G(Y)} = \cEP{Y}{P_G(Z)} = \cEP{\cEP{Y}{Z}}{P_G(Z)} \\
       &\; =  \cEP{Z}{P_G(Z)} = \cEP{Z}{P_G(Y)}. 
  \end{align*}
  
  Since $\sigma(P_G(Y)) \subset \sigma(Y)$, we have
  $X\in \mathcal{X}(Y)$.  By Theorem~\ref{th:2}, the law of $(X,Y)$ is
  an optimal plan if and only if $X \in P_G(Y)$.  From
  Lemma~\ref{lem:3} and the assumption that $Y\not\in \Sigma_1(P_G)$,
  we deduce that $P_G(Y)$ takes values in closed convex sets. Hence,
  \begin{displaymath}
    X \in P_G(Y) \iff \ip{s}{X} \leq g_s(Y) \set \sup_{z\in P_G(Y)}
    \ip{s}{z}, \quad s\in \real{d},   
  \end{displaymath}
  where in showing the implication $\impliedby$, we choose an
  exceptional null set that works for a dense countable subset of
  $s\in \real{d}$.
  
  Let $s\in \real{d}$. By Lemma~\ref{lem:5}, $g_s$ is
  $\sigma(P_G)$-measurable and so, $g_s(Y)$ is
  $\sigma(P_G(Y))$-measurable. As $Z\in P_G(Y)$, we have
  $\ip{s}{Z} \leq g_s(Y)$. It follows that
  \begin{displaymath}
    \ip{s}{X} = \cEP{\ip{s}{Z}}{P_G(Y)} \leq  \cEP{g_s(Y)}{P_G(Y)} =
    g_s(Y). 
  \end{displaymath}
  Hence, $X\in P_G(Y)$, as required.

  To conclude the proof we only have to verify~\eqref{eq:7}.  The
  optimal map $X$ is $Z$-measurable, because it is $P_G(Y)$-measurable
  and $P_G(Y) = P_G(Z)$. By the tower property,
  \begin{displaymath}
    \cEP{Z}{X} = \cEP{\cEP{Y}{Z}}{X} = \cEP{Y}{X} = X. 
  \end{displaymath}
  Since $X$ and $Z$ take values in the $S$-monotone set $G$, we have
  \begin{displaymath}
    S(Z-X,Z-X) \geq 0. 
  \end{displaymath}
  Using the tower property and the optimality of $X$ and $Z$, we
  obtain
  \begin{align*}
    \EP{S(Z-X,Z-X)} & =  \EP{S(Z-X,Z)} = \EP{S(Z-X,Y)} \\
                    & =  \EP{S(Z,Y)}  -
                      \EP{S(X,Y)}  = 0.
  \end{align*}
  It follows that $S(Z-X,Z-X) = 0$.
\end{proof}

Following~\cite{KramSirb:22b}, we now state sufficient conditions for
the assertions of Theorem~\ref{th:3} to hold that do not involve a
dual minimizer.  Let $j\in \braces{1,\dots,d}$ and $C$ be a compact
set in $\real{d}$ such that $y^j=1$, $y\in C$.  For $x \in \real{d}$,
we denote by $x^{-j}$ its sub-vector without the $j$th coordinate:
\begin{displaymath}
  x^{-j} \set 
  \cbraces{x^1,\dots,x^{j-1},x^{j+1},\dots,x^{d}}
  \in \real{d-1}. 
\end{displaymath}
We write $\mathcal{H}^{j}_{C}$ for the family of functions $h=h(x)$ on
$\real{d}$ having the decomposition:
\begin{displaymath}
  h(x) = x^j + g_1(x^{-j}) - g_2(x^{-j}), \quad
  x\in \real{d},    
\end{displaymath}
where the functions $g_1$ and $g_2$ on $\real{d-1}$ are convex, have
linear growth:
\begin{displaymath}
  \abs{g_i(x)} \leq K(1+\abs{x}), \quad x\in \real{d}, \; i=1,2, 
\end{displaymath}
for some constant $K>0$, and $\nabla h(x) \in C$, whenever the
functions $g_1$ and $g_2$ are differentiable at $x^{-j}$.

Let $h\in \mathcal{H}^j_C$ and $H$ be the zero-level set of the
composition function $h\circ S$, that is,
$H \set \descr{x\in \real{d}}{h(Sx) = 0}$.  Lemma~4.4
in~\cite{KramSirb:22b} shows that the surface $H$ has at every point
an $S$-normal vector in $C$.  In the spirit of \citet[Definition~6.3,
p.~199]{RockWets:98}, a vector $w\in \real{d}$ is called
$S$-\emph{regular normal to $H$ at $x\in H$} if
\begin{displaymath}
  \limsup_{\substack{H\ni y\to x \\ y\not=x}} \frac{S(w,y-x)}{\abs{y-x}}\leq 0.  
\end{displaymath}
A vector $w\in \real{d}$ is called $S$-\emph{normal to $H$ at $x$} if
there exist $x_n\in H$ and an $S$-regular normal vector $w_n$ to $H$
at $x_n$, such that $x_n\rightarrow x$ and $w_n\rightarrow w$.

\begin{Theorem}
  \label{th:4}
  Let $S\in \msym{d}{m}$, $Y \in \mathcal{L}_2(\real{d})$, and denote
  $\nu\set \law{Y}$.  Assume that the convex hull of the support of
  $\nu$ has a non-empty interior:
  \begin{equation}
    \label{eq:8}
    \interior{\conv{\supp{\nu}}}\not=\emptyset, 
  \end{equation}
  and that
  \begin{enumerate}[label = {\rm (A\arabic{*})}, ref={\rm
      (A\arabic{*})}]
  \item \label{item:13} $h(SY) \not =0$ for every function
    $h\in \mathcal{H}^j_C$ with index $j\in \braces{1,\dots,d}$ and
    compact set
    \begin{equation}
      \label{eq:9}
      C \subset\descr{x\in \real{d}}{x^j=1, \; S(x,x) >0}. 
    \end{equation}
  \end{enumerate}
  Then $Y\not \in \Sigma_1(P_G)$ for every $G \in \mset{S}$ such that
  $\EP{\psi_G(Y)}<\infty$, and the assertions of Theorem~\ref{th:3}
  hold.
\end{Theorem}

\begin{proof}
  Let $G\in \mset{S}$ be such that
  $\EP{\psi_G(Y)}<\infty$. Condition~\eqref{eq:8} implies that the
  domain of $\psi_G$ has a non-empty interior.  Theorem~4.6
  in~\cite{KramSirb:22b} yields functions
  $h_n \in \mathcal{H}_{C_n}^{j_n}$ with indexes
  $j_n\in \braces{1,\dots,d}$ and compact sets
  \begin{displaymath}
    C_n \subset\descr{x\in \real{d}}{x^{j_n}=1, \; S(x,x) >0}, \quad
    n\geq 1, 
  \end{displaymath}
  such that
  \begin{displaymath}
    \Sigma_1(P_G) \subset \bigcup_n \descr{x\in \real{d}}{h_n(Sx) = 0}. 
  \end{displaymath}
  The result readily follows.
\end{proof}

Condition~\eqref{eq:9} means that the zero-level set $H$ of the
composition function $h\circ S$ has at every point an $S$-normal
vector $w$ with $w^j=1$ which is strictly $S$-positive.  For $d=2$ and
the standard $S$ from Example \ref{ex:1}, the latter property is
equivalent to $H$ having at every point a normal vector, in the
classical Euclidean sense, that lies in the strictly positive
orthant. In this case, $H$ is the graph of a strictly decreasing $c-c$
function and Theorem~\ref{th:4} improves Theorem~4.5
in~\cite{KramXu:22}, where $\nu(F)=0$ for every graph $F$ of a
strictly decreasing Lipschitz function.

\section{Uniqueness of optimal maps and plans}
\label{sec:uniq-solut-map}

Theorem~\ref{th:6} states explicit conditions for the uniqueness of
optimal maps and plans.  We start with an intermediate result.

\begin{Theorem}
  \label{th:5}
  Let $S\in \msym{d}{m}$, $Y \in \mathcal{L}_2(\real{d})$, and denote
  $\nu \set \law{Y}$. Let $G\in \mset{S}$ be an optimal set
  for~\eqref{eq:3} and assume that $Y \not\in \Sigma(P_G)$, that is,
  the projection $P_G(Y)$ is (almost surely) single-valued. Then
  \begin{displaymath}
    X \set P_G(Y) \ind{Y\not\in \Sigma(P_G)}
  \end{displaymath}
  is the unique optimal map for~\eqref{eq:1} and the law of $(X,Y)$ is
  the unique optimal plan for~\eqref{eq:2}.
\end{Theorem}

\begin{proof}
  Lemma~\ref{lem:5} shows that $\Sigma(P_G)$ is
  $\sigma(P_G)$-measurable. Hence, $X$ is $\sigma(P_G(Y))$-measurable
  and then is also $Y$-measurable.  Let $\gamma$ be an optimal plan
  for~\eqref{eq:2}.  By extending, if necessary, the probability
  space, we can find a $d$-dimensional random variable $Z$ such that
  $\gamma = \law{Z, Y}$. As $x = \gamma(y|x)$, we have
  $Z = \cEP{Y}{Z}$. Theorem~\ref{th:2} shows that $Z\in P_G(Y)$. Since
  $P_G$ is single-valued, $Z = X$.
\end{proof}

Following~\cite{KramSirb:22b}, we denote by $\mathcal{E}_0(S)$ the
collection of Borel sets $D\subset \real{d}$ with the property that
for any $\delta >0$, there are functions
$h_n\in \mathcal{H}^{j_n}_{C_n}$ with indexes
$j_n\in \braces{1,\dots,d}$ and compact sets
\begin{displaymath}
  C_n \subset\descr{x\in \real{d}}{x^{j_n}=1, \; 0\leq S(x,x)\leq
    \delta}, \quad n\geq 1, 
\end{displaymath}
such that
\begin{displaymath}
  D\subset \bigcup_n\descr{x\in \real{d}}{h_n(Sx)=0}.
\end{displaymath}
In particular, the family $\mathcal{E}_0(S)$ contains all hyperplanes
\begin{displaymath}
  H(w,c) \set \descr{y\in \real{d}}{S(y,w) = c} = \descr{y\in
    \real{d}}{\ip{Sy}{w} -c = 0}, 
\end{displaymath}
where $c\in \real{}$ and the unique $S$-normal vector $w$ to $H(w,c)$
is $S$-\emph{isotropic} in the sense that $\abs{w}>0$ and
$S(w,w) = 0$. Thus, in Theorem~\ref{th:6}, condition~\ref{item:14} is
weaker than~\ref{item:15} stated for $m=1$.

Geometrically, for any $\delta >0$, the elements of $\mathcal{E}_0(S)$
are covered by a countable number of $c-c$ surfaces, that have at
every point $S$-normal vectors $w$ with $w^j=1$ and
$0\leq S(w,w)\leq \delta$.  Heuristically, for small $\delta>0$, these
$S$-normal vectors are \emph{almost} $S$-isotropic.

\begin{Theorem}
  \label{th:6}
  Let $S\in \msym{d}{m}$, $Y \in \mathcal{L}_2(\real{d})$, and denote
  $\nu\set \law{Y}$. Assume~\eqref{eq:8} and that
  \begin{enumerate}[label = {\rm (A\arabic{*})}, ref={\rm
      (A\arabic{*})}]
    \setcounter{enumi}{1}
  \item \label{item:14} If $m=1$, then for every $S$-isotropic vector
    $w$, the random variable $S(Y,w)$ has a continuous cumulative
    distribution function.
  \item \label{item:15} If $m>1$, then $Y\notin D$, that is,
    $\nu(D)=0$, for every $D\in \mathcal{E}_0(S)$.
  \end{enumerate}
  Then $Y\not \in \Sigma_0(P_G)$ for every $G \in \mset{S}$ such that
  $\EP{\psi _G(Y)}<\infty$.

  If, in addition, \ref{item:13} holds, then $Y\not \in \Sigma(P_G)$
  for every $G \in \mset{S}$ such that $\EP{\psi _G(Y)}<\infty$ and
  the assertions of Theorem~\ref{th:5} hold.
\end{Theorem}

\begin{proof}
  Let $G\in \mset{S}$ be such that $\EP{\psi_G(Y)}<\infty$.  In view
  of Theorem~\ref{th:4}, we only need to show that
  $Y\not\in \Sigma_0(P_G)$.  From~\eqref{eq:8} we deduce that the
  domain of $\psi_G$ has a non-empty interior, which enables us to use
  Theorems~4.8 and~4.7 in~\cite{KramSirb:22b}.

  Let $m=1$.  Theorem~4.8 in~\cite{KramSirb:22b} yields $S$-isotropic
  $w_n\in \real{d}$ and constants $c_n\in \real{}$, $n\geq 1$, such
  that
  \begin{displaymath}
    \Sigma _0(P_G)\subset \bigcup _n\descr{x\in \real{d}}{S(x,w_n)=c_n}.
  \end{displaymath}
  Under~\ref{item:14}, $Y\not\in \Sigma_0(P_G)$.

  Let $m>1$. For $j\in \braces{1,\dots, d}$, we define
  $\overline{\Sigma}^j_0(P_G)$ as the set of $x\in \dom{P_G}$ such
  that $S(y_1-y_2,y_1-y_2)=0$ for \emph{some} $y_1,y_2 \in P_G(x)$
  with $y_1^j \not = y_2^j$. Clearly,
  \begin{displaymath}
    \Sigma_0(P_G) \subset \bigcup_{j=1}^d
    \overline{\Sigma}^j_0(P_G).
  \end{displaymath}
  Fix $\delta >0$ and $j\in \braces{1, \dots, d}$.  By Theorem~4.7
  in~\cite{KramSirb:22b}, there exist compact sets
  \begin{displaymath}
    C_n \subset\descr{x\in \real{d}}{x^j=1, \;  0\leq S(x,x) \leq \delta}
  \end{displaymath}
  and functions $h_n\in \mathcal{H}^j_{C_n}$, $n\geq 1$, such that
  \begin{displaymath}
    \overline{\Sigma}^j_0(P_G) \subset \bigcup_n \descr{x\in
      \real{d}}{h_n(Sx) = 0}.  
  \end{displaymath}
  By Lemma \ref{lem:5}, $\Sigma _0(P_G)$ is Borel measurable. Hence,
  $\Sigma _0(P_G)\in \mathcal{E}_0(S)$ and then, under~\ref{item:15},
  $Y\not\in \Sigma_0(P_G)$.
\end{proof}

For $d=2$ and the standard $S$ from Example~\ref{ex:1},
Theorem~\ref{th:6} improves Theorem~4.6 in~\cite{KramXu:22}. In this
case, \ref{item:14} means that each of the components $Y_i$, $i=1,2$,
has a continuous cumulative distribution function.

Given~\eqref{eq:8}, Theorem \ref{th:6} yields the uniqueness of
optimal maps and plans provided that $\nu \set \law{Y}$ does not
charge any $c-c$ surface whose normals in the $S$-space are either
strictly positive or
\begin{enumerate}[label = {\rm (\alph{*})}, ref={\rm (\alph{*})}]
  \setcounter{enumi}{0}
\item isotropic if the index $m=1$;
\item nonnegative and almost isotropic if $m>1$.
\end{enumerate}
In particular, the maps and plans are unique, if $\nu$ does not charge
any $c-c$ surface having, at every point, a nonnegative normal in the
$S$-space. Theorem~\ref{th:7} contains a formal statement.

\begin{Theorem}
  \label{th:7}
  Let $S\in \msym{d}{m}$, $Y \in \mathcal{L}_2(\real{d})$, and denote
  $\nu\set \law{Y}$. Assume~\eqref{eq:8} and that
  \begin{enumerate}[label = {\rm (A\arabic{*})}, ref={\rm
      (A\arabic{*})}]
    \setcounter{enumi}{3}
  \item\label{item:16} $h(SY) \not =0$ for every
    $h\in \mathcal{H}^j_C$ with index $j\in \braces{1,\dots,d}$ and
    compact set
    \begin{displaymath}
      C \subset\descr{x\in \real{d}}{x^j=1, \; S(x,x) \geq 0}.
    \end{displaymath}
  \end{enumerate}
  Then $Y\not \in \Sigma(P_G)$ for every $G \in \mset{S}$ such that
  $\EP{\psi _G(Y)}<\infty$ and the assertions of Theorem~\ref{th:5}
  hold.
\end{Theorem}

\begin{proof}
  The result is an immediate corollary of Theorem~\ref{th:6},
  because~\ref{item:16} implies~\ref{item:13}, \ref{item:14},
  and~\ref{item:15}.
\end{proof}

We point out that~\ref{item:16} and then also~\ref{item:13},
\ref{item:14}, and~\ref{item:15} hold if $\nu$ has a density with
respect to the Lebesgue measure on $\real{d}$.

\section{Equilibrium with insider}
\label{sec:equil-with-insid}

In this section, we study a multi-asset version of equilibrium with
insider from~\cite{RochVila:94}. Our results generalize those
from~\cite{KramXu:22} obtained for the model with one stock.

We consider a single-period financial market with $m$ stocks. The
terminal prices of the stocks are random and represented by
$V \in \lsp{2}{m}$. The initial prices are the result of the
interaction between \emph{noise traders}, an \emph{insider}, and a
\emph{market maker}:
\begin{enumerate}
\item The noise traders place an order for $U$ stocks;
  $U \in \lsp{2}{m}$.
\item The insider knows the {value} of $Y\set (U,V)$ and places an
  order for $Q$ stocks.  The {trading strategy} $Q$ and the total
  order $R=Q + U$ are $m$-dimensional $Y$-measurable random variables.
\item The market maker observes only the total order $R=Q+U$.  He
  quotes the price $f(R)$ according to a {pricing rule}
  $\map{f}{\real{m}}{\real{m}}$, which is a Borel function.
\end{enumerate}

\begin{Definition}
  \label{def:1}
  Let $U,V,R\in \lsp{2}{m}$, $\map{f}{\real{m}}{\real{m}}$ be a Borel
  function, denote $Y\set(U,V)$, and assume that $R$ is
  $Y$-measurable.  The pair $(R,f)$ is called a $Y$-\emph{equilibrium}
  if
  \begin{enumerate}[label = {\rm (\alph{*})}, ref={\rm (\alph{*})}]
  \item \label{item:17} Given the total order $R$, the pricing rule
    $f$ is \emph{efficient} in the sense that
    \begin{equation}
      \label{eq:10}
      f(R) = \cEP{V}{R}. 
    \end{equation}
  \item \label{item:18} Given the pricing rule $f$, the trading
    strategy $Q = R-U$ \emph{maximizes} the insider profit:
    \begin{displaymath}
      Q \in \argmax\descr{q\in \real{m}}{\ip{q}{V-f(q+U)}},  
    \end{displaymath}
    or equivalently,
    \begin{equation}
      \label{eq:11}
      R \in \argmax\descr{r\in \real{m}}{\ip{r-U}{V-f(r)}}.
    \end{equation}
  \end{enumerate}
\end{Definition}

It is natural to expect that an equilibrium pricing function $f$ is
\emph{monotone}:
\begin{displaymath}
  \ip{u-r}{f(u)-f(r)}\geq 0, \quad u,r\in \real{m}. 
\end{displaymath}
In this case, if $\omega \in \Omega$ is such that
$Y(\omega) \in \graph{f}$ or, equivalently,
$V(\omega) = f(U(\omega))$, then an optimal strategy for the insider
is not to trade: $Q(\omega)=0$. Theorem~\ref{th:10}\ref{item:28} shows
that every equilibrium pricing function $f$ is monotone provided that
the support of the law of $Y$ is the whole space $\real{2m}$.

Hereafter, in this section, $S$ is the \emph{standard} matrix from
Example~\ref{ex:1}:
\begin{displaymath}
  S((r,s),(u,v)) = \ip{s}{u} + \ip{r}{v}, \; r,s,u,v \in \real{m}. 
\end{displaymath}
The following theorem shows a close connection between the equilibrium
problem~\eqref{eq:10}--\eqref{eq:11} and the map problem~\eqref{eq:1}
for such $S$.  We point out that a multifunction
$\mmap{T}{\real{m}}{\real{m}}$ is monotone in the classical sense:
\begin{displaymath}
  \ip{u-r}{v-s} \geq 0, \quad v\in T(u), \; s\in T(r), 
\end{displaymath}
if and only if its graph is $S$-monotone.
 
\begin{Theorem}
  \label{th:8}
  Let $U,V \in \lsp{2}{m}$, $Y \set (U,V)$, $\nu \set \law{Y}$, and
  $S\in \msym{2m}m$ be the standard matrix.  The following conditions
  are equivalent:
  \begin{enumerate}[label = {\rm (\alph{*})}, ref={\rm (\alph{*})}]
  \item \label{item:19} There exists a $Y$-equilibrium with a monotone
    pricing function.
  \item \label{item:20} There exist an optimal map $X$
    for~\eqref{eq:1} and an optimal set $G$ for~\eqref{eq:3} such that
    the law of $(X,Y)$ is an optimal plan for~\eqref{eq:2} and the
    projection of $G$ on the first $m$ coordinates is the whole space
    $\real{m}$:
    \begin{equation}
      \label{eq:12}
      \proj{x^1,\dots,x^m}{G} \set \descr{u\in \real{m}}{
        (u,v) \in G\text{~for~some~}v\in \real{m}} = \real{m}. 
    \end{equation}
  \end{enumerate}
  Under these conditions, for any optimal set $G$ for~\eqref{eq:3}
  satisfying~\eqref{eq:12}, there exists a $Y$-equilibrium $(R,f)$
  such that $\graph{f} \subset G$ and $X\set (R,f(R))$ is an optimal
  map for~\eqref{eq:1}.
\end{Theorem}

\begin{Remark}
  \label{rem:1}
  Let $S\in \msym{2m}m$ be the standard matrix. Every $G\in \mset{S}$
  is the graph of a uniquely defined maximal monotone multifunction
  $\mmap{T}{\real{m}}{\real{m}}$. We have
  \begin{displaymath}
    \dom{T} \set \descr{u\in \real{m}}{T(u) \not = \emptyset} =
    \proj{x^1,\dots,x^m} G. 
  \end{displaymath}
  Thus, $\proj{x^1,\dots,x^m} G = \real{m}$ if and only if
  $\dom{T} = \real{m}$.
\end{Remark}

The proof of Theorem~\ref{th:8} is split into lemmas. We extend the
definitions of the Fitzpatrick-type function $\psi_F$ and the
multifunction $P_F$ to any closed set $F\subset \real{2m}$:
\begin{align*}
  \psi_F(y) & \set \sup_{x\in F} \cbraces{S(x,y) - \frac12 S(x,x)}, \\
  P_F(y) & \set \argmax_{x\in F} \cbraces{S(x,y) - \frac12
           S(x,x)}, \; y\in \real{2m}.  
\end{align*}
From the definition of the subdifferential, we immediately obtain
\begin{equation}
  \label{eq:13}
  P_F(y) \subset \descr{x\in F}{Sx \in \partial \psi_F(y)}, \quad y \in
  \real{2m}.  
\end{equation}
Given an equilibrium $(R,f)$ and taking $F \set \cl{\graph{f}}$, the
insider profit can be written as
\begin{displaymath}
  \ip{R-U}{V-f(R)} = \psi_F(Y) - \frac12 S(Y,Y)   
\end{displaymath}
and the profit-maximization condition~\eqref{eq:11} as
$(R,f(R)) \in P_F(Y)$.

\begin{Lemma}
  \label{lem:6}
  Let $S\in \msym{2m}{m}$ be the standard matrix and $F$ be a closed
  set in $\real{2m}$ such that $\proj{x^1,\dots,x^m}{F} = \real{m}$.
  Then
  \begin{equation}
    \label{eq:14}
    \psi_F(y)\geq \frac 12 S(y,y), \quad
    y\in \real{2m}.  
  \end{equation}
\end{Lemma}

\begin{proof}
  Let $y = (u,v)$, where $u, v\in \real{m}$. Taking $x= (u,w) \in F$
  and using the fact that $S(x-y,x-y) = 2 \ip{u-u}{v-w} = 0$, we
  obtain
  \begin{displaymath}
    \frac12 S(y,y) = S(x,y) - \frac12 S(x,x) \leq \psi_F(y),
  \end{displaymath}
  and the result follows.
\end{proof}

\begin{Lemma}
  \label{lem:7}
  Let $S\in \msym{2m}{m}$ be the standard matrix and $F$ be a closed
  $S$-monotone set in $\real{2m}$ satisfying~\eqref{eq:14}.  Then
  \begin{displaymath}
    H \set  \descr{x}{\psi_F(x)=\frac 12 S(x,x)}\in \mset{S}.
  \end{displaymath}
  For $G\in \mset{S}$, we have
  \begin{displaymath}
    F\subset G\iff G=H  \iff \psi _G=\psi _F\iff \psi _G\leq\psi _F.
  \end{displaymath}
\end{Lemma}

\begin{proof}
  For any $G\in \mset{S}$, Theorem~A.2 in \cite{KramSirb:24} shows
  that
  \begin{displaymath}
    \psi_G(x) = \frac 12 S(x,x), \; x \in G, \text{~and~}  \psi_G(x)
    > \frac12 S(x,x),\; x \in \real{2m}\setminus G. 
  \end{displaymath}
  Let $G\in\mset{S}$ be such that $\psi_G\leq \psi _F$.  As $F$ is a
  closed $S$-monotone set,
  \begin{displaymath}
    \psi_F(x)=\frac 12 S(x,x), \quad x\in F.
  \end{displaymath}
  It follows that $F \subset G$ and then, trivially, that
  $\psi_F \leq \psi_G$, $\psi_F = \psi_G$, and $G=H$.
  
  By Theorem A.3 in \cite{KramSirb:24}, the bound~\eqref{eq:14}
  implies the existence of $G\in \mset{S}$ such that
  $\psi_G\leq \psi_F$. Consequently, $H=G\in \mset{S}$, $F\subset H$,
  and $\psi_F=\psi_H$.

  If $G\in \mset{S}$ and $F\subset G$, then
  $\psi _H=\psi_F \leq \psi_G$ and
  \begin{displaymath}
    G = \descr{x}{\psi_G(x)=\frac 12 S(x,x)} \subset
    \descr{x}{\psi_H(x)=\frac 12 S(x,x)} = H. 
  \end{displaymath}
  The maximality property of $G\in \mset{S}$ implies that $G=H$.
\end{proof}

The next lemma provides a construction of an optimal map from an
equilibrium.

\begin{Lemma}
  \label{lem:8}
  Let $U,V \in \lsp{2}{m}$, $Y \set (U,V)$, $(R,f)$ be a
  $Y$-equilibrium, $W\set \cEP{U}{R}$, and $F \set \cl{\graph{f}}$.
  Let $\nu \set \law{Y}$ and $S \in \msym{2m}m$ be the standard
  matrix. Then
  \begin{enumerate}[label = {\rm (\alph{*})}, ref={\rm (\alph{*})}]
  \item \label{item:21} $X\set (W,f(R))$ is an optimal map
    for~\eqref{eq:1} and the law of $(X,Y)$ is an optimal plan
    for~\eqref{eq:2}.
  \item \label{item:22} There exists a $G\in \mset{S}$ such that
    $\psi_F\geq \psi_G$. Any such set $G$ is optimal for \eqref{eq:3}.
  \item \label{item:23} A set $G\in \mset{S}$ is optimal
    for~\eqref{eq:3} if and only if $\psi_F(Y) = \psi_G(Y)$.
  \item \label{item:24} If $f$ is monotone, then there exists a unique
    $G\in \mset{S}$ containing $F$.  The set $G$ is optimal for
    \eqref{eq:3} and $\proj{x^1,\dots,x^m} G = \real{m}$.
  \end{enumerate}
\end{Lemma}

\begin{proof}
  \ref{item:21} + \ref{item:22}: By the definition of equilibrium, $X$
  is $Y$-measurable and $X = \cEP{Y}{R} = \cEP{Y}{X}$. Thus,
  $X \in \mathcal{X}(Y)$. The profit-maximization
  condition~\eqref{eq:11} is equivalent to
  $Z\set (R,f(R)) \in P_F(Y)$. Consequently,
  \begin{gather*}
    \psi_F(Y)  = S(Z,Y) - \frac12 S(Z,Z), \\
    \EP{\psi_F(Y)|Z} = \EP{\psi_F(Y)|R}  = S(Z,X) - \frac12 S(Z,Z) = \frac12 S(X,X),
  \end{gather*}
  where the last equality holds because
  \begin{displaymath}
    S(Z-X,Z-X) = S((W-R,0), (W-R,0))= 0.
  \end{displaymath} 

  Clearly, the projection of $F = \cl\graph{f}$ on the first $m$
  coordinates is the whole space $\real{m}$. By Lemma~\ref{lem:6}, the
  function $\psi_F$ has the lower bound~\eqref{eq:14}. Theorem~A.3
  in~\cite{KramSirb:24} then yields $G\in \mset{S}$ such that
  $\psi_F \geq \psi_G$, proving the first part of~\ref{item:22}.
  By~Theorem~A.2 in~\cite{KramSirb:24},
  \begin{displaymath}
    \psi_G(x) = \frac 12 S(x,x), \; x \in G, \text{~and~}  \psi_G(x)
    > \frac12 S(x,x),\; x \in \real{2m}\setminus G. 
  \end{displaymath}
  Using Jensen's inequality, we obtain
  \begin{displaymath}
    \EP{\psi_G(Y)} \geq \EP{\psi_G(X)} \geq \frac 12
    \EP{S(X,X)} = \EP{S(X,Y) - \frac12 S(X,X)}. 
  \end{displaymath}
  As $\psi _G\leq \psi _F$ and $\EP{\psi _F(Y)}=\frac 12 \EP{S(X,X)}$,
  we deduce that
  \begin{displaymath}
    \psi_F(Y) = \psi_G(Y) \text{~and~}  \psi_G(X) = \frac12 S(X,X)
  \end{displaymath}
  and then that $X\in G$ and $\psi_G(Y) = S(X,Y) - \frac12 S(X,X)$,
  which is equivalent to saying that $X \in P_G(Y)$.
  Item~\ref{item:21} and the second part of~\ref{item:22} follow now
  from Theorem~\ref{th:2}.
 
  \ref{item:23}: Fix $G$ as in \ref{item:22}. Let $H\in \mset{S}$. If
  $\psi_F(Y) = \psi_H(Y)$, then $\psi_G(Y) = \psi_H(Y)$ and the
  optimality of $H$ readily follows. Conversely, if $H$ is optimal,
  then $X \in P_H(Y)$ by Theorem~\ref{th:2}.  Hence,
  $X \in P_G(Y) \cap P_H(Y)$ and we again obtain that
  $\psi_G(Y) = \psi_H(Y) = \psi_F(Y)$.

  \ref{item:24}: As $F$ is $S$-monotone and
  $\proj{x^1,\dots,x^m} F = \real{m}$, Lemmas~\ref{lem:6}
  and~\ref{lem:7} yield a unique $G\in \mset{S}$ containing $F$ or,
  equivalently, such that $\psi_G=\psi _F$.  By~\ref{item:23}, $G$ is
  a dual minimizer. Clearly, $\proj{x^1,\dots,x^m} G = \real{m}$.
\end{proof}

In the next lemma, an equilibrium is constructed from an optimal map
$X$ and an optimal set $G$ such that
$\proj{x^1,\dots,x^m} G = \real{m}$.

\begin{Lemma}
  \label{lem:9}
  Let $U,V,R,W \in \lsp{2}{m}$, $Y \set (U,V)$, and $X\set (R,W)$. Let
  $\nu \set \law{Y}$ and $S\in \msym{2m}{m}$ be the standard
  matrix. Assume that there exists an optimal set $G$ for~\eqref{eq:3}
  such that $\proj{x^1,\dots,x^m} G = \real{m}$.  If $X$ is an optimal
  map for~\eqref{eq:1} and the law of $(X,Y)$ is an optimal plan
  for~\eqref{eq:2}, then there exists a monotone Borel function
  $\map{f}{\real{m}}{\real{m}}$ such that
  \begin{equation}
    \label{eq:15}
    f(R) = \cEP{W}{R} \text{~and~} F\set \cl\graph{f}\subset G. 
  \end{equation}
  In this case, $(R,f(R))$ is an optimal map for~\eqref{eq:1}, the
  pair $(R,f)$ is a $Y$-equilibrium, and $G$ is the only set in
  $\mset{S}$ containing $F$.
\end{Lemma}

\begin{proof}
  Let $\map{g}{\real{m}}{\real{m}}$ be a Borel function such that
  $g(R) = \EP{W|R}$. Using the tower property, we deduce that
  \begin{align*}
    Z \set (R,g(R)) &= \cEP{X}{Z} = \cEP{Y}{Z}, \\
    S(X-Z,X-Z) & = S((0,W-f(R)),(0,W-f(R)))  = 0, \\    
    \cEP{S(X,X)}{Z} &= \cEP{S(X-Z,X-Z)}{Z} + S(Z,Z) =  S(Z,Z), \\
    \EP{S(X,Y)} & = \EP{S(X,X)} = \EP{S(Z,Z)} = \EP{S(Z,Y)}. 
  \end{align*}
  As $Z$ is $R$-measurable, it is also $X$ and $Y$-measurable. Thus,
  $Z \in \mathcal{X}(Y)$.  Since the law of $(X,Y)$ is an optimal
  plan, the law of $(Z,Y)$ is an optimal plan as well. In particular,
  $Z$ is an optimal map.

  By Theorem~\ref{th:2}, $Z \in P_G(Y)$. In particular,
  $Z = (R, g(R)) \in G$. Let $\map{h}{\real{m}}{\real{m}}$ be a Borel
  function, whose graph is contained in $G$. The existence of such
  function follows from Remark~\ref{rem:1} and
  Lemma~\ref{lem:4}\ref{item:9}. Then
  \begin{displaymath}
    f(x) \set g(x) \ind{(x,g(x)) \in G} + h(x) \ind{(x,g(x)) \not\in
      G}, \quad x\in \real{m}, 
  \end{displaymath}
  is a Borel function satisfying~\eqref{eq:15}.
 
  For \emph{any} Borel function $f$ satisfying~\eqref{eq:15}, we have
  that $Z = (R, f(R))$ (almost surely, as usual). We have already
  shown that $Z$ is an optimal map. As $Z \in P_G(Y)$ and
  $F \set \cl\graph{f}\subset G$, we obtain that $Z \in P_F(Y)$, which
  is exactly the profit maximizing condition~\eqref{eq:11}. Thus,
  $(R,f)$ is a $Y$-equilibrium. Finally, Lemma~\ref{lem:7} shows that
  $G$ is the only maximal monotone set containing $F$.
\end{proof}

\begin{proof}[Proof of Theorem~\ref{th:8}]
  The implication \ref{item:19}~$\implies$~\ref{item:20} is proved in
  Lemma~\ref{lem:8}.  The reverse implication
  \ref{item:20}~$\implies$~\ref{item:19} and the last assertion of the
  theorem follow from Lemma~\ref{lem:9}.
\end{proof}

We now state explicit sufficient conditions for the existence of an
equilibrium with monotone pricing function. For an $m$-dimensional
random variable $U$, we write $\supp{U}$ for the support of the law of
$U$.

\begin{Theorem}
  \label{th:9}
  Let $U,V \in \lsp{2}{m}$, $Y \set (U,V)$, $\nu \set \law{Y}$, and
  $S\in \msym{2m}m$ be the standard matrix. Assume~\eqref{eq:8},
  \ref{item:13}, and that $\conv{\supp{U}} = \real{m}$ or $V$ is
  bounded.  Then there exists an optimal set $G$ for~\eqref{eq:3} and
  a $Y$-equilibrium $(R,f)$ such that $f$ is monotone,
  $X\set (R,f(R))$ is an optimal map for~\eqref{eq:1}, the law of
  $(X,Y)$ is an optimal plan for~\eqref{eq:2}, and
  $\graph{f} \subset G$.
\end{Theorem}

The proof of Theorem~\ref{th:9} relies on the following criteria for
the existence of an optimal set $G$ for~\eqref{eq:3} such that
$\proj{x^1,\dots,x^m} G = \real{m}$.

\begin{Lemma}
  \label{lem:10}
  Let $U,V \in \lsp{2}{m}$, $Y \set (U,V)$, and $S\in \msym{2m}m$ be
  the standard matrix.
  \begin{enumerate}[label = {\rm (\alph{*})}, ref={\rm (\alph{*})}]
  \item \label{item:25} If $\conv{\supp{U}} = \real{m}$, then
    $\proj{x^1,\dots,x^m}{G} = \real{m}$ for every optimal set $G$
    for~\eqref{eq:3}.
  \item \label{item:26} If $V$ is bounded, then there exists an
    optimal set $G$ for~\eqref{eq:3} such that
    $\proj{x^1,\dots,x^m}{G} = \real{m}$.
  \end{enumerate}
\end{Lemma}

\begin{proof}
  \ref{item:25}: Let $G$ be an optimal set for~\eqref{eq:3}.  From
  $\EP{\psi _G(Y)}<\infty$ and the convexity of $\dom{\psi _G}$ we
  obtain
  \begin{equation}
    \label{eq:16}
    \proj{x^1, \dots, x^m}{\dom \psi_G}=\real{m}.
  \end{equation}
  As $G\in \mset{S}$, by~\cite[Theorem~12.41, p.~555]{RockWets:98}
  and~Remark~\ref{rem:1}, we know that the projection of $G$ on the
  first $m$-coordinates is \emph{nearly convex}. In other words, there
  exists a convex set $D$ such that
  \begin{displaymath}
    D\subset \proj{x^1, \dots, x^m}{G} \subset \cl D.
  \end{displaymath}
  It remains to be shown that
  $D=\proj{x^1, \dots, x^m}{G} = \cl{D} = \real{m}$, which is
  equivalent to
  \begin{displaymath}
    \dist(u,D) \set \inf_{r\in D} \norm{u-r} = 0, \quad u\in
    \real{m}. 
  \end{displaymath}
  We fix $u\in \real{m}$. In view of~\eqref{eq:16}, there exists
  $v \in \real{m}$ such that $y = (u,v) \in \dom{\psi_G}$. For every
  $t>0$, Minty's parametrization of $G\in \mset{S}$, Theorems~12.12
  and~12.15 in \cite[p.~539--540]{RockWets:98}, yields unique $r(t)$
  and $s(t)$ in $\real{m}$ such that $x(t) = (r(t),s(t)) \in G$ and
  \begin{displaymath}
    t u + v = t r(t) + s(t). 
  \end{displaymath}
  Then $v - s(t) = t(r(t) - u)$ and
  \begin{align*}
    \psi_G(y) - \frac12 S(y,y)
    & \geq S(x(t),y) - \frac12
      \cbraces{S(x(t),x(t)) + S(y,y)} \\
    &  = \frac12 S(x(t)-y,y-x(t)) =
      \ip{r(t) - u}{v-s(t)} \\
    & = t \norm{r(t) - u}^2 \geq t \dist^2(u,D).
  \end{align*}
  Taking $t\to \infty$, we deduce that $\dist(u,D)=0$.

  \ref{item:26}: Let $\gamma$ be an optimal plan
  for~\eqref{eq:2}. Passing, if necessary, to a larger probability
  space we can assume that $\gamma = \law{X,Y}$ for
  $X\in \lsp{2}{2m}$. We write $X = (R,W)$, where
  $R,W \in \lsp{2}{m}$.  As $W=\cEP{V}{X}$ and $V$ is bounded, $W$ is
  also bounded. Let $B$ be a closed convex set in $\real{m}$ whose
  interior contains the values of $W$. We denote by $N_B$ the normal
  cone to $B$:
  \begin{align*}
    N_B(v) & \set  \descr{u\in\real{m}}{\ip{u}{s-v}\leq 0, \ s\in B},
             \quad v\in B,\\ 
    N_B(v) & \set \emptyset, \quad v\notin B. 
  \end{align*}
 
  Let $H$ be an optimal set for \eqref{eq:3}. By Theorem~\ref{th:2},
  $X = (R,W) \in H$. As $W\in \interior{B}$, we have
  \begin{displaymath}
    \cbraces{\proj{x^{m+1}, \dots, x^{2m}}{H}} \cap \interior{B}\not=
    \emptyset. 
  \end{displaymath}
  This condition allows us to use the \emph{truncation} result for
  maximal monotone multifunctions from \cite[Example~12.45(a),
  p.~557]{RockWets:98}. According to this result, the set
  \begin{displaymath}
    G\set \descr{(u+r,v)\in \real{2m}}{(u,v)\in H,  r \in N_B(v)}
  \end{displaymath}
  belongs to $\mset{S}$, coincides with $H$ on
  $\real{m}\times \interior{B}$, and
  $\proj{x^{1},\dots,x^{m}}{G} = \real{m}$. Moreover, for
  $u\in \real{m}$ and $v \in B$,
  \begin{align*}
    \psi_G(u,v) &=\sup_{(r,s)\in H, q\in N_B(s)} 
                  \cbraces{\ip{r+q}{v} + \ip{u}{s}- \ip{r+q}{s}}\\
                & =\sup_{(r,s)\in H, q \in N_B(s)} 
                  \cbraces{\ip{r}{v}+\ip{u}{s}-\ip{r}{s}+ \ip{q}{v-s}}\\
                & \leq \sup_{(r,s)\in H} 
                  \cbraces{\ip{r}{v}+\ip{u}{s}-\ip{r}{s}}
                  = \psi_H(u,v).
  \end{align*}
  As $Y = (U,V) \in \real{m}\times B$, the optimality of $H$ implies
  that $\psi_G(Y) = \psi_H(Y)$. Hence, $G$ is an optimal set for
  \eqref{eq:3}.
\end{proof}

\begin{proof}[Proof of Theorem~\ref{th:9}]
  Theorem~\ref{th:4} yields an optimal map $X$ for~\eqref{eq:1} such
  that the law of $\cbraces{X,Y}$ is an optimal plan
  for~\eqref{eq:2}. Using Lemma~\ref{lem:10} we obtain an optimal set
  $G$ for~\eqref{eq:3} such that $\proj{x^1,\dots,x^m} G = \real{m}$.
  The rest of the proof follows from Theorem~\ref{th:8}.
\end{proof}

Next, we provide uniqueness criteria for the equilibrium. For a
function $\map{f}{\real{m}}{\real{m}}$, we define its continuity set
as
\begin{displaymath}
  C(f) \set \descr{u\in \real{m}}{f
    \text{~is continuous at~}u}.
\end{displaymath}

\begin{Theorem}
  \label{th:10}
  Let $U,V \in \lsp{2}{m}$, $Y \set (U,V)$, $\nu \set \law{Y}$, and
  $S\in \msym{2m}m$ be the standard matrix.  Let $(R,f)$ and $(Q,g)$
  be $Y$-equilibria.
  \begin{enumerate}[label = {\rm (\alph{*})}, ref={\rm (\alph{*})}]
  \item \label{item:27} The $Y$-equilibria yield the same insider
    profit:
    \begin{displaymath}
      \ip{R-U}{V-f(R)}=  \ip{Q-U}{V-g(Q)}.
    \end{displaymath}
  \item \label{item:28} If $\supp{\nu} = \real{2m}$, then there exists
    a unique $H\in \mset{S}$ optimal for~\eqref{eq:3}.  The graphs of
    $f$ and $g$ are contained in $H$.  In particular, $f$ and $g$
    coincide on their common continuity set:
    \begin{displaymath}
      f(u) = g(u), \quad u \in C(f) = C(g). 
    \end{displaymath}
  \item \label{item:29} If~\eqref{eq:8}, \ref{item:13}, \ref{item:14},
    and \ref{item:15} hold, then the equilibrium prices are unique:
    $f(R)=g(Q)$.  If, in addition, $f$ and $g$ are monotone or
    $Y \in \interior{\supp{\nu}}$, then the equilibrium total orders
    are also unique: $R=Q$. In this case, $X\set (R,f(R))$ is the
    unique optimal map for~\eqref{eq:1} and the law of $(X,Y)$ is the
    unique optimal plan for~\eqref{eq:2}.
  \end{enumerate}
\end{Theorem}

As in Theorem~\ref{th:7}, the stronger condition~\ref{item:16} can be
used in Theorem~\ref{th:10}\ref{item:29} instead of \ref{item:13},
\ref{item:14}, and \ref{item:15}.

We divide the proof of Theorem~\ref{th:10} into lemmas.  Let $H$ be
the dual minimizer from Theorem~\ref{th:10}\ref{item:28} and the
multifunction $\mmap{T}{\real{m}}{\real{m}}$ be such that
$\graph{T} = H$. The next lemma, which is essentially a special case
of \citet[Theorem~3]{Qi:83}, shows that the common continuity set $C$
of $f$ and $g$ coincides with the set where $T$ is
single-valued. Therefore, according to \citet[Theorem~1]{Zara:73}, $C$
is an $F_{\sigma}$-set of full Lebesgue measure.

\begin{Lemma}
  \label{lem:11}
  Let $S\in \msym{2m}m$ be the standard matrix and
  $\map{f}{\real{m}}{\real{m}}$ be a monotone function.  Define the
  multi-valued function $\mmap{T}{\real{m}}{\real{m}}$ by
  \begin{displaymath}
    T(u)\set \conv \descr {v\in \real{m}}{
      f(u_n)\rightarrow v\text{~for some sequence~} u_n\rightarrow u},
    \quad u\in \real{m}. 
  \end{displaymath}
  Then, $\graph{T}$ is the unique set in $ \mset{S}$ containing
  $\graph{f}$ and
  \begin{displaymath}
    C(f) 
    = 
    \descr{u\in \real{m}}{T
      \text{~is single-valued  at~}u}.
  \end{displaymath}
\end{Lemma}

\begin{proof}
  We recall that $\map{f}{\real{m}}{\real{m}}$ is monotone if and only
  if $F\set \cl{\graph{f}}$ is $S$-monotone.  Lemmas \ref{lem:6} and
  \ref{lem:7} show that there exists a unique $G\in \mset{S}$
  containing $F$. Let $\mmap{A}{\real{m}}{\real{m}}$ be the monotone
  multifunction whose graph coincides with $G$.

  Fix $u\in \real{m}$. Since
  $\graph{f} \subset \graph{A} = G \in \mset{S}$, the set $A(u)$ is
  closed, convex, and contains $T(u)$. Having $\real{m}$ as its
  domain, $A$ (and then also $f$) is locally bounded at $u$,
  \cite[Corollary~12.38, p.~554]{RockWets:98}. It follows that $A(u)$
  and $T(u)$ are convex compacts.

  Consider $v\notin T(u)$.  We separate $v$ strongly from $T(u)$, that
  is, we choose $r\in \real{m}$ and $\epsilon>0$ such that
  $\ip{r}{v-s}\geq 2 \epsilon$ for any $s\in T(u)$.  As $f$ is bounded
  in a neighborhood of $u$ and $T(u)$ contains all the cluster points
  of $f$ at $u$, there exists $\delta>0$ such that
  \begin{displaymath}
    \ip{(u+rt)-u}{v-f(u+rt)}=t\ip{r}{v-f(u+rt)}\geq t\epsilon>0, \quad
    t\in (0,\delta). 
  \end{displaymath}
  Since $G$ is $S$-monotone and $\graph{f}\subset G$, we obtain that
  $(u,v)\notin G$, or, equivalently, that $v\notin A(u)$. Thus,
  $T(u) = A(u)$ and $G = \graph{T}$.

  The description of $C(f)$ at the end of the lemma follows from the
  local boundedness of $f$ and the standard compactness argument.
\end{proof}

\begin{Lemma}
  \label{lem:12}
  Let $S\in \msym{d}{m}$, $G \in \mset{S}$, and $F$ be a closed set in
  $\real{d}$. Then
  \begin{displaymath}
    F\subset G \quad \iff \quad \psi_F \leq \psi_G.
  \end{displaymath}  
\end{Lemma}

\begin{proof}
  The implication $\implies$ is straightforward. If
  $y\in F\setminus G$, then the maximal $S$-monotonicity of $G$
  implies the existence of $x\in G$ such that ${S(x-y,x-y) < 0}$.
  Since $\psi_G(x) = \frac12 S(x,x)$, we obtain
  \begin{align*}
    \psi_F(x) & \geq S(x,y) - \frac12 S(y,y) = \frac12 \cbraces{S(x,x) -
                S(x-y,x-y)} \\
              &> \frac 12 S(x,x) = \psi_G(x)  
  \end{align*}
  and arrive to a contradiction.
\end{proof}
\begin{proof}[Proof of Theorem~\ref{th:10}] We denote
  $F\set \cl{\graph{f}}$ and $G\set \cl{\graph{g}}$.

  \ref{item:27}: The insider profit for the $Y$-equilibrium $(R,f)$,
  can be written as
  \begin{displaymath}
    \ip{R-U}{V-f(R)} = \psi_F(Y) - \frac12 S(Y,Y),
  \end{displaymath}
  and similarly for $(Q,g)$.  By~Lemma~\ref{lem:8}\ref{item:23},
  $\psi_F(Y)=\psi_H(Y)=\psi_G(Y)$, where $H\in \mset{S}$ is any
  optimal set for~\eqref{eq:3}.
    
  \ref{item:28}: As $\supp{\nu} = \real{2m}$, Theorem~3.5 in
  \cite{KramSirb:24} yields the unique $H\in\mset{S}$ optimal for
  \eqref{eq:3}.  By~Lemma~\ref{lem:8}\ref{item:23},
  $\psi_F(Y)=\psi_H(Y)=\psi_G(Y)$.  The convex functions $\psi_F$,
  $\psi_H$, and $\psi_G$ are then finite and coincide on $\real{2m}$.
  Lemma \ref{lem:12} yields $F \cup G \subset H$.
  
  Consider $\mmap{T}{\real{m}}{\real{m}}$ such that $H=\graph{T}$.  By
  Lemma \ref{lem:11},
  \begin{displaymath}
    C(f) = \descr{u\in \real{m}}{T \text{~is single-valued
        at~}u}=C(g). 
  \end{displaymath}
  Consequently, $f(u)=g(u)$ for $u\in C(f)=C(g)$.

  \ref{item:29}: We denote $Z\set (R,f(R))$ and recall that the
  profit-maximizing condition~\eqref{eq:11} can be written as
  $Z\in P_F(Y)$.  Let $H$ be an optimal set for~\eqref{eq:3}.  By
  Lemma~\ref{lem:8}\ref{item:23}, $\psi_F(Y)=\psi_H(Y)$.
  Theorem~\ref{th:6} shows that $Y\notin \Sigma (P_H)$ and that the
  unique optimal map and plan are given by
  $M\set P_H(Y) \ind{Y\not\in \Sigma(P_H)}$ and the law of $(M,Y)$,
  respectively.
  
  We denote $W\set \cEP{U}{R}$.  By Lemma~\ref{lem:8}\ref{item:21},
  $X\set (W,f(R))$ is an optimal map for \eqref{eq:1}. Consequently,
  $X= (W,f(R))=M$.  Using a similar argument for $(Q,g)$, we obtain
  $f(R) = g(Q)$.
  
  If $f$ is monotone, then in view of Lemma~\ref{lem:8}\ref{item:24},
  we can choose the optimal $H$ such that $F\subset H$.  From
  $Z\in P_F(Y)$ and $\psi_F(Y)=\psi_H(Y)$, we obtain $Z\in P_H(Y)$. As
  $Y\notin \Sigma (P_H)$, we conclude that $Z = (R,f(R))=M$.
  Similarly, $(Q,g(Q)) = M$.
     
  As $\psi_F(Y) = \psi_H(Y)$, the closed convex functions $\psi_F$ and
  $\psi_H$ are finite and coincide on the open set
  $D\set \interior{\supp{\nu}}$. Clearly,
  $\partial \psi_F = \partial \psi_H$ on $D$. Let
  $y \in D\setminus \Sigma(P_H)$. By the properties of Fitzpatrick
  functions, \cite[Theorem~A.4]{KramSirb:24}, $\psi_H$ is
  differentiable at $y$ and $S^{-1}\nabla \psi_H(y)$ is the only
  element of $P_H(y)$.  In view of~\eqref{eq:13}, $P_F(y)$ is either
  empty or coincides with $P_H(y)$. If now $Y \in D$, then actually
  $Y \in D\setminus \Sigma(P_H)$ and we obtain that $P_F(Y) = P_H(Y)$
  and $(R,f(R)) = M$.  Similarly, $(Q,g(Q)) = M$.
\end{proof}

We finish the section with a Gaussian example, where the equilibrium
pricing function $f$ is linear.  The constant matrix $A = \nabla f$ in
Theorem~\ref{th:11}, the sensitivity of the price to the total trading
order, is the multi-dimensional version of Kyle's lambda
from~\cite{Kyle:85}. This matrix is the unique positive-definite
solution of the algebraic Riccati equation~\eqref{eq:17}.

We say that an $m\times m$ matrix $A$, possibly non-symmetric, is
positive-definite, and write $A>0$, if $\ip{r}{Ar}>0$ for every
$r\in \real{m}$.  For a \emph{symmetric} matrix $C>0$, we denote by
$C^{\alpha}$ the power of $C$ with exponent $\alpha \in \real{}$. The
matrix $C^{\alpha}$ is symmetric and positive-definite. Any powers of
$C$ commute. We denote by $B^T$ the transpose of a matrix $B$.

\begin{Theorem}
  \label{th:11}
  Let $S\in \msym{2m}{m}$ be the standard matrix and
  $U,V\in \lsp{2}{m}$ be such that $Y\set (U,V)$ has a non-degenerate
  centered Gaussian distribution with block covariance matrix
  structure
  \begin{gather*}
    \Sigma_{uu}\set \EP{UU^T}>0, \quad \Sigma_{vv}\set \EP{VV^T}>0,\\
    \Sigma_{uv}\set \EP{UV^T}, \quad \Sigma _{vu}\set
    \EP{VU^T}=\Sigma^T _{uv}.
  \end{gather*}
  \begin{enumerate}[label = {\rm (\alph{*})}, ref={\rm (\alph{*})}]
  \item \label{item:30} The non-symmetric algebraic Riccati equation
    \begin{equation}
      \label{eq:17}
      A\Sigma _{uu}A+\cbraces{A\Sigma _{uv}-\Sigma_{vu}A}=\Sigma_{vv}
    \end{equation}
    has the unique $m\times m$ matrix solution $A>0$.
  
  \item \label{item:31} The $Y$-equilibrium $(R,f)$ is unique and has
    the form:
    \begin{displaymath}
      R\set  \cbraces{A+A^T}^{-1}(A^TU+V),\quad 
      f(u)\set Au, \ u \in \real{m}.
    \end{displaymath}
    The random variable $X\set \cbraces{R,f(R)}$ is the unique optimal
    map for~\eqref{eq:1}, the law of $\cbraces{X,Y}$ is the unique
    optimal plan for~\eqref{eq:2}, and $\graph{f}$ is the unique
    optimal set for~\eqref{eq:3}.
  
  \item \label{item:32} The symmetric equation
    $\Lambda \Sigma_{uu}\Lambda=\Sigma_{vv}$ has the unique
    $m\times m$ positive-definite matrix solution
    \begin{displaymath}
      \Lambda \set \Sigma _{uu}^{-\frac12} 
      \cbraces{\Sigma _{uu}^{\frac12} \Sigma_{vv}
        \Sigma _{uu}^{\frac12} }^{\frac12}
      \Sigma _{uu}^{-\frac12}.
    \end{displaymath}
    The matrix $A$ solving~\eqref{eq:17} is symmetric if and only if
    $\Lambda \Sigma_{uv}=\Sigma_{vu}\Lambda$, that is, if and only if
    $\Lambda \Sigma_{uv}$ is symmetric.  In this case,
    \begin{displaymath}
      A=\Lambda, \quad  R=\frac 12\cbraces{U+\Lambda^{-1}V}.
    \end{displaymath}
  \end{enumerate}
\end{Theorem}

The proof of the theorem relies on Theorems~4.3 and~4.5
in~\cite{KramSirb:24} and the following simple observation. For
$S\in \msym{d}{m}$, a closed set $G$ in $\real{d}$ is called
\emph{strictly $S$-monotone} if
\begin{displaymath}
  S(x-y,x-y) > 0, \quad x,y \in G, \; x\not = y.  
\end{displaymath}
If, in addition, $G\in \mset{S}$, then we say that $G$ is maximal
strictly $S$-monotone.

\begin{Lemma}
  \label{lem:13}
  Let $S\in \msym{2m}{m}$ be the standard matrix. A closed set $G$ in
  $\real{2m}$ is a linear maximal strictly $S$-monotone subspace of
  $\real{2m}$ if and only if $G = \graph{f}$, where $f(u) = Au$,
  $u\in \real{m}$, and $A$ is an $m\times m$ positive-definite matrix.
\end{Lemma}

\begin{proof}
  $\impliedby$: The function $f(u) = Au$ is strictly monotone:
  \begin{displaymath}
    \ip{f(u) - f(v)}{u-v} =  \ip{A(u-v)}{u-v} > 0, \quad u\not= v,
  \end{displaymath}
  so $\graph{f}$ is strictly $S$-monotone.  Being defined on the whole
  $\real{m}$, this linear function is maximal monotone, according to
  \cite[Example~12.7, p.~535]{RockWets:98}.  In other words,
  $\graph{f}\in \mset{S}$.

  $\implies$: We denote $D \set \proj{x^1,\dots,x^m}{G}$ and observe
  that $D$ is a linear subspace of $\real{m}$.  Being linear strictly
  $S$-monotone, $G$ is the graph of a linear strictly monotone
  function $\map{f}{D}{\real{m}}$. By Minty's parametrization of
  $G\in \mset{S}$, the linear function $u\to u + f(u)$ is a bijection
  between $D$ and $\real{m}$. It follows that $D=\real{m}$.  Being a
  linear function on $\real{m}$, $f$ can be written as $f(u) = Au$,
  $u \in \real{m}$, for an $m\times m$ matrix $A$. As $f$ is strictly
  monotone, $A>0$.
\end{proof}

\begin{proof}[Proof of Theorem~\ref{th:11}]
  We fix an $m\times m$ matrix $A>0$, define the linear function
  $f_A(r) = Ar$, $r\in \real{m}$, and denote ${G_A} \set \graph
  f_A$. Lemma~\ref{lem:13} shows that $G_A \in \mset{S}$. By direct
  computations, the $S$-projection of $y=(u,v)$ on $G_A$ has the form:
  \begin{align*}
    P_{{G_A}}(y) & = \argmin_{x \in {G_A}} S(y-x,y-x) = \argmin_{(w,Aw) \in {G_A}}
                   \ip{u-w}{v-Aw} \\
                 & = (r,Ar),  \quad
                   r = \cbraces{A+A^T}^{-1}(A^Tu+v).
  \end{align*}
  An $S$-normal vector at $x = (r,Ar) = P_{G_A}(y)$ is given by
  \begin{displaymath}
    y - x = (u-r,v-Ar) = \cbraces{u-r, -A^T(u-r)}.
  \end{displaymath}

  Let $R\set \cbraces{A+A^T}^{-1}(A^TU+V)$ and
  $X\set P_{G_A}(Y) = (R,AR)$.  In view of the Gaussian structure,
  $X= \cEP{Y}{X}$ if and only if
  \begin{displaymath}
    X= (R,AR)\text{~and~}Y-X=(U-R, -A^T(U-R))
  \end{displaymath}
  are independent.  This is equivalent to the independence of $R$ and
  $U-R$ and then to the independence of
  \begin{displaymath}
    (A+A^T)R= A^TU+V\text{~and~}(A+A^T)(U-R)
    =AU-V.  
  \end{displaymath}
  Due to the Gaussian structure, the last independence property can be
  written as
  \begin{align*}
    0  & = \EP{(AU-V)(A^TU+V)^T} = \EP{(AU-V)(U^TA + V^T)} \\
       & = A \Sigma_{uu} A + A \Sigma_{uv}  - \Sigma_{vu}A -
         \Sigma_{vv}, 
  \end{align*}
  and thus, is equivalent to~\eqref{eq:17}. We have shown that
  \begin{displaymath}
    X\set P_{G_A}(Y)= \cEP{Y}{X} \iff 
    A\text{~solves~} \eqref{eq:17}.
  \end{displaymath}
  
  Using the equivalence of items~\ref{item:4} and~\ref{item:5} in
  Theorem~\ref{th:2}, we deduce
  \begin{equation}
    \label{eq:18}
    G_A\text{~ is optimal for~}\eqref{eq:3}\iff
    A \text{~solves~} \eqref{eq:17}.
  \end{equation}
  However, by~\cite[Theorems~4.3 and~4.5]{KramSirb:24}, the dual
  problem~\eqref{eq:3} has only one solution $G$ and this solution is
  a linear maximal strictly $S$-monotone subspace of
  $\real{2m}$. Lemma~\ref{lem:13} yields $A>0$ such that $G =
  G_A$. Clearly, such $A$ is unique.  In view of~\eqref{eq:18}, $A$ is
  the only positive-definite solution of the matrix equation
  \eqref{eq:17}.
  
  As $Y$ has a strictly positive density on $\real{2m}$, all
  conditions of Theorem~\ref{th:10} are satisfied. It follows that
  $(R,f)$ defined in~\ref{item:31} is the unique $Y$-equilibrium, that
  $X=(R,AR)$ is the unique optimal map for \eqref{eq:1} and the law of
  $(X,Y)$ is the unique optimal plan for \eqref{eq:2}. We have
  proved~\ref{item:30} and~\ref{item:31}.
  
  For~\ref{item:32}, we observe that $\Lambda$ is symmetric
  positive-definite. We check directly that
  $\Lambda \Sigma_{uu}\Lambda=\Sigma_{vv}$. Thus, $\Lambda$
  solves~\eqref{eq:17} if $\Sigma_{uv}=0$. As we have shown, $\Lambda$
  is the only positive-definite solution in this case.
  
  If $A$ is the positive-definite solution of~\eqref{eq:17}, then its
  transpose solves
  \begin{displaymath}
    A^T\Sigma_{uu}A^T-\cbraces{A^T\Sigma _{uv}-\Sigma_{vu}A^T}=\Sigma_{vv}.
  \end{displaymath}
  Comparing to~\eqref{eq:17}, the equality $A=A^T$ holds if and only
  if $A\Sigma_{uv}=\Sigma_{vu}A$, in which case $A$ and $\Lambda$
  solve the same equation. Consequently, if $A$ is symmetric, we have
  $A=\Lambda$ and $\Lambda \Sigma_{uv}=\Sigma _{vu}\Lambda$.
  Conversely, if $\Lambda \Sigma_{uv}$ is symmetric, then $\Lambda$
  solves~\eqref{eq:17} and thus, coincides with $A$.
\end{proof}

Theorem~\ref{th:11}\ref{item:32} describes those covariance matrices
$\Sigma_{uv}$ for which the pricing matrix $A$ is the same as the
pricing matrix $\Lambda$ for the uncorrelated case.  This extends a
similar observation about Kyle's lambda from~\cite{Kyle:85}
and~\cite{RochVila:94} for the model with just one stock.

\appendix

\section{Uniform approximation by maps}
\label{sec:linfty-density-monge}

We continue to identify random variables on a probability space
$(\Omega, \mathcal{F}, \mathbb{P})$ if they differ only on a set of
measure zero and interpret relations between them in the
$\as{\mathbb{P}}$ sense.  In particular, if $X$ and $Z$ are random
variables taking values in Polish (complete separable metric) spaces
$(\mathbb{S}_1, \rho_1)$ and $(\mathbb{S}_2,\rho_2)$, respectively,
then $X$ is $Z$-measurable if and only if $X=f(Z)$ ($\as{\mathbb{P}}$)
for a Borel function $\map{f}{\mathbb{S}_2}{\mathbb{S}_1}$.

\begin{Theorem}
  \label{th:12}
  Let $X$ and $Y$ be random variables on
  $(\Omega, \mathcal{F}, \mathbb{P})$ taking values in Polish spaces
  $(\mathbb{S}_1, \rho_1)$ and $(\mathbb{S}_2,\rho_2)$,
  respectively. If the law of $Y$ is atomless, then for every
  $\epsilon>0$ there exists a random variable $Z=Z(\epsilon)$ taking
  values in $(\mathbb{S}_2,\rho_2)$ such that $\law{Z} = \law{Y}$,
  $\rho_2\cbraces{Y,Z}\leq \epsilon$, and $X$ is $Z$-measurable.
\end{Theorem}

\begin{proof}
  Let $\epsilon >0$. We take a dense sequence $y_n \in \mathbb{S}_2$,
  $n=1,2,\dots$, and define
  \begin{displaymath}
    B_n\set \descr{z\in
      \mathbb{S}_2}{\rho_2(z,y_n)\leq\frac{\epsilon}{2}}, \quad
    D_n\set B_n\setminus \bigcup _{i=1}^{n-1}B_i, \quad n\geq 1.  
  \end{displaymath}
  By keeping only the terms having strictly positive probability, we
  obtain sets $D_n\subset\mathbb{S}_2$, $n= 1,2,\dots$, that are
  \emph{mutually disjoint} and such that
  \begin{displaymath}
    Y \in \bigcup_{n=1}^{\infty} D_n, \quad \PP{Y\in D_n} > 0, \quad \diam
    (D_n):=\sup_{x,y\in D_n}\rho_2(x,y)\leq \varepsilon. 
  \end{displaymath}
  For each $n\geq 1$ we apply Lemma~\ref{lem:14} using the probability
  measure
  \begin{displaymath}
    \mathbb{P}_n(A) \set \mathbb{P}(A|Y\in D_n) = \frac{\PP{A\cap
        \braces{Y\in D_n}}}{\PP{Y\in D_n}}, \quad A\in \mathcal{F}, 
  \end{displaymath}
  instead of the original measure $\mathbb{P}$.  As a result, we
  obtain a random variable $Z_n$ taking values in $D_n$ and a Borel
  function $\map{f_n}{\mathbb{S}_2}{\mathbb{S}_1}$ such that $Z_n$ and
  $Y$ have the same law under $\mathbb{P}_n$ and
  $\mathbb{P}_n(X = f_n(Z_n)) = 1$.  For the random variable $Z$ and
  the Borel function $\map{f}{\mathbb{S}_2}{\mathbb{S}_1}$ such that
  \begin{displaymath}
    Z(\omega) = Z_n(\omega)\text{~if~}Y(\omega)\in D_n,
    \quad f( y) = f_n ( y)\text{~if~}y\in D_n, \quad n\geq 1,  
  \end{displaymath}
  we have $\Law (Y)=\Law (Z)$, $X=f(Z)$,
  and~$\rho_2\cbraces{Y,Z} \leq \epsilon$.
\end{proof}

\begin{Lemma}
  \label{lem:14}
  Let $X$ and $Y$ be random variables taking values in Polish spaces
  $(\mathbb{S}_1, \rho_1)$ and $(\mathbb{S}_2,\rho_2)$,
  respectively. If the law of $Y$ is atomless, then there exists a
  random variable $Z$ taking values in $(\mathbb{S}_2,\rho_2)$ such
  that $\Law (Z)=\Law (Y)$ and $X$ is $Z$-measurable.
\end{Lemma}

\begin{proof}
  Replacing $X$ with $(X,Y)$ we can assume from the start that the law
  of $X$ is atomless.  Using bijections
  $\map{g}{\mathbb{S}_1}{\real{}}$ and
  $\map{h}{\mathbb{S}_2}{\real{}}$ that are Borel measurable together
  with their inverses, we can also replace $X$ with $g(X)$ and $Y$
  with $h(Y)$.  For the existence of such bijections we refer
  to~\citet[Theorem~13.1.1]{Dudl:02}.  Thus, we assume that both $X$
  and $Y$ are real valued and have continuous cumulative distribution
  functions.

  We denote by $F_X$ and $Q_X$ the cumulative distribution and
  quantile functions of $X$:
  \begin{align*}
    F_X(x) &\set \PP{X\leq x}, \quad x\in \real{}, \\
    Q_X(u) &\set \min \descr{x\in \real{}}{u\leq F_X(x)}, \quad u\in (0,1). 
  \end{align*}
  As $F_X$ is continuous, the random variable $U\set F_X(X)$ has
  uniform distribution on $(0,1)$ and $X = Q_X(U)$. Similarly,
  $V\set F_Y(Y)$ has uniform distribution on $(0,1)$ and $Y = Q_Y(V)$,
  where $F_Y$ and $Q_Y$ are the cumulative distribution and quantile
  functions of $Y$. Setting
  \begin{displaymath}
    Z \set Q_Y(U) = Q_Y(F_X(X)),
  \end{displaymath}
  we deduce that $Z$ has the same law as $Y$, that $U = F_Y(Z)$, and
  $X = Q_X(F_Y(Z)))$.
\end{proof}

\begin{Remark}
  \label{rem:2}
  Let $X$ and $Y$ be random variables on
  $(\Omega, \mathcal{F}, \mathbb{P})$ taking values in $\real{d}$ and
  assume that the law of $Y$ is atomless. The conditional law of $X$
  given $Y$ describes the randomized transport from $Y$, the
  ``origin'', to $X$, the ``target''. A classical question is to
  approximate $(X,Y)$ by $(X',Y')$, where the transport from $Y'$ to
  $X'$ is deterministic, that is, $X'$ is $Y'$-measurable.

  Theorem~\ref{th:12} shows that $(X,Y)$ can be \emph{pointwise
    uniformly} approximated by the elements of the family
  \begin{displaymath}
    \mathcal{C}_1(X,Y) \set \descr{(X,Z)}{\law{Z} = \law{Y} \text{~and
        $X$ is $Z$-measurable}}, 
  \end{displaymath}
  where the \emph{target} $X$ is kept frozen.

  In a more traditional approach, see~\cite{Prat:07}
  and~\cite{BeigLack:18} among the others, the approximating family is
  \begin{displaymath}
    \mathcal{C}_2(X,Y) \set \descr{(Z,Y)}{\law{Z} = \law{X} \text{~and
        $Z$ is $Y$-measurable}}, 
  \end{displaymath}
  where the \emph{origin} $Y$ remains unchanged. It has been shown
  that the law of $(X,Y)$ can be \emph{weakly} approximated by the
  laws of the elements of $\mathcal{C}_2(X,Y)$.  In general, $(X,Y)$
  can not be pointwise approximated by the elements of
  $\mathcal{C}_2(X,Y)$. For instance, if $X$ and $Y$ are independent
  and $(Z,Y) \in \mathcal{C}_2(X,Y)$, then $Z$ and $X$ are independent
  ($Z$ is a function of $Y$) identically distributed. Thus, the law of
  $W\set \abs{X-Z}$ is exactly the same for all such $Z$ and $W\not=0$
  as soon as $X$ is not a constant.
\end{Remark}

\bibliographystyle{plainnat} \bibliography{../bib/finance}

\end{document}